\documentclass[a4paper, 11pt, reqno]{amsart} 
\usepackage[utf8]{inputenc}
\usepackage[numbers, sort&compress]{natbib}
\usepackage{appendix}
\usepackage[pdftoolbar=true, pdfmenubar=true, pdftitle={}, pdfsubject={}, pdfauthor={}, pdfkeywords={}, pdfcreator={}, pdfproducer={}, bookmarks=false, bookmarksopenlevel=section]{hyperref}
\usepackage[paper=a4paper,left=25mm,right=25mm,top=25mm,bottom=25mm]{geometry}

\usepackage[babel]{csquotes}

\usepackage[english]{babel}	 
\usepackage{graphicx} 
\usepackage{wrapfig} 
\usepackage{amsmath}
\usepackage{amsthm}		
\usepackage{eucal}
\usepackage{amsfonts, dsfont}								
\usepackage{amssymb}	
\usepackage{bbm}
\usepackage{braket}	
\usepackage{enumerate}

\usepackage{tikz}
\usetikzlibrary{quotes,angles,patterns,backgrounds,decorations.pathreplacing}
\usepackage{pgffor}
\usepackage{wrapfig}
\newbox\mybox

\numberwithin{equation}{section}

\newtheorem{theorem}{Theorem}[section]
\newtheorem{lemma}[theorem]{Lemma}

\newtheorem{proposition}[theorem]{Proposition}
\newtheorem{corollary}[theorem]{Corollary}

\theoremstyle{remark}
\newtheorem{remark}[theorem]{Remark}

\theoremstyle{definition}

\newtheorem*{assumption*}{\assumptionnumber}
\providecommand{\assumptionnumber}{}

\makeatletter
\renewcommand\@biblabel[1]{\textbf{#1.}} 

\usepackage{color}
\DeclareMathOperator{\sgn}{sgn}

\DeclareMathOperator{\tr}{tr}

\DeclareMathOperator{\dist}{dist}

\newcommand{\ang}[1]{\langle #1 \rangle}

\newcommand{\interior}[1]{%
  {\kern0pt#1}^{\mathrm{o}}%
}
\makeatletter
\DeclareRobustCommand\widecheck[1]{{\mathpalette\@widecheck{#1}}}
\def\@widecheck#1#2{%
    \setbox\z@\hbox{\m@th$#1#2$}%
    \setbox\tw@\hbox{\m@th$#1%
       \widehat{%
          \vrule\@width\z@\@height\ht\z@
          \vrule\@height\z@\@width\wd\z@}$}%
    \dp\tw@-\ht\z@
    \@tempdima\ht\z@ \advance\@tempdima2\ht\tw@ \divide\@tempdima\thr@@
    \setbox\tw@\hbox{%
       \raise\@tempdima\hbox{\scalebox{1}[-1]{\lower\@tempdima\box
\tw@}}}%
    {\ooalign{\box\tw@ \cr \box\z@}}}
\makeatother

\begin{document}

\title{\textbf{A Szeg\H{o} limit theorem\\ for translation-invariant operators on polygons}} 

\author{\textsc{Bernhard Pfirsch}} 


\date{\today} 

\newcommand{\Addresses}{{
  \bigskip

\textsc{Department of Mathematics,
University College London,
Gower Street,
London
WC1E 6BT, UK.}\par\nopagebreak
  \textit{e-mail:} \texttt{bernhard.pfirsch.15@ucl.ac.uk}
}}
\begin{abstract}
We prove Szeg\H{o}-type trace asymptotics for translation-invariant operators on polygons. More precisely, consider a Fourier multiplier $A=\mathcal{F}^\ast \sigma \mathcal{F}$ on $\mathsf{L}^2(\mathbb{R}^2)$ with a sufficiently decaying, smooth symbol $\sigma:\mathbb{C}\to\mathbb{C}$. Let $P\subset \mathbb{R}^2$ be the interior of a polygon and, for $L\geq 1$, define its scaled version $P_L:=L\cdot P$. Then we study the spectral asymptotics for the operator $A_{P_L}=\chi_{P_L}A\chi_{P_L}$, the spatial restriction of $A$ onto $P_L$: for entire functions $h$ with $h(0)=0$ we provide a complete asymptotic expansion of $\tr h(A_{P_L})$ as $L\to\infty$. These trace asymptotics consist of three terms that reflect the geometry of the polygon. If $P$ is replaced by a domain with smooth boundary, a complete asymptotic expansion of the trace has been known for more than 30 years. However, for polygons the formula for the constant order term in the asymptotics is new. In particular, we show that each corner of the polygon produces an extra contribution; as a consequence, the constant order term exhibits an anomaly similar to the heat trace asymptotics for the Dirichlet Laplacian.
\end{abstract}

\maketitle 


\section{Introduction}

Let $A$ be a bounded and translation-invariant operator on $\mathsf{L}^2(\mathbb{R}^d)$. In other words, consider a Fourier multiplier
\begin{align*}
A=A(\sigma)=\mathcal{F}^\ast \sigma \mathcal{F}
\end{align*}
with a bounded, complex-valued \textit{symbol} $\sigma\in\mathsf{L}^\infty(\mathbb{R}^d)$. Here, the Fourier transform $\mathcal{F}$ is chosen to be unitary on $\mathsf{L}^2(\mathbb{R}^d)$. For any measurable set $\Omega\subseteq\mathbb{R}^d$, introduce the spatial restriction of the operator $A$ onto $\Omega$,
\begin{align*}
A_{\Omega}:=A_{\Omega}(\sigma):=\chi_{\Omega}\mathcal{F}^\ast \sigma \mathcal{F}\chi_{\Omega},
\end{align*}
where $\chi_\Omega$ denotes the characteristic function for the set $\Omega$ and both $\chi_\Omega$ and $\sigma$ are interpreted as multiplication operators on $\mathsf{L}^2(\mathbb{R}^d)$. In analogy to the one-dimensional case, we refer to such an operator $A_\Omega$ as \textit{(multidimensional) truncated Wiener-Hopf operator}.
Throughout this paper, the variable $L\geq 1$ is used as a scaling parameter and
\begin{align*}
\Omega_L:=L\cdot\Omega
\end{align*}
denotes the scaled version of the set $\Omega$. For the sake of discussion, assume that the set $\Omega\subset\mathbb{R}^d$ is bounded and the symbol $\sigma$ of $A$ belongs to the class of smooth and rapidly decreasing functions, i.e.~$\sigma\in\mathcal{S}(\mathbb{R}^d)$. Moreover, let $h:\mathbb{C}\to\mathbb{C}$ be an entire function with $h(0)=0$. Under these assumptions, it is well-known that the operator $h(A_{\Omega})$ is trace class and the function $h$ is also called \textit{test function}. If in addition $\partial \Omega$ is smooth, then \cite{Widom1985} provides a complete asymptotic expansion of 
\begin{align}\label{eq:trhAOmegaL}
\tr h(A_{\Omega_L})
\end{align}
as $L\to\infty$. More precisely, for any $K\geq -d$ there exist constants $\mathcal{B}_j=\mathcal{B}_j(\Omega,h,\sigma)$ such that
\begin{align}\label{eq:trasymptsmoothboundary}
\tr h(A_{\Omega_L})=\sum\limits_{j=-K}^d L^j\mathcal{B}_j + o(L^{-K}),
\end{align}
as $L\to\infty$. These trace asymptotics for truncated Wiener-Hopf operators can be seen as a continuous multi-dimensional analogue of Szeg\H{o}'s famous limit theorem for Toeplitz matrices, see \cite{Szego1952}. While 
\begin{align}\label{eq:Bd}
\mathcal{B}_d=\frac{|\Omega|}{(2\pi)^d}\int\limits_{\mathbb{R}^d}d\xi\,(h\circ \sigma)(\xi)
\end{align}
only depends on $\Omega$ through its volume $|\Omega|$, the coefficients $\mathcal{B}_j$ for $j\leq d-1$ contain geometric information on the boundary $\partial \Omega$: $\mathcal{B}_{d-1}$ arises from a hyperplane approximation at each point of $\partial\Omega$ and $\mathcal{B}_{d-2}$ contains the curvature and the second fundamental form of $\partial\Omega$, see also \cite{Roccaforte1984}. As a general principle, the coefficient $\mathcal{B}_{d-k}$ depends on $C^k$-attributes of $\partial\Omega$; more precise formulae in terms of the geometric content are collected in  \cite{Roccaforte2013}.

The asymptotics of \eqref{eq:trhAOmegaL} have also been studied intensively  for non-smooth symbols $\sigma$, even though, for $d\geq 2$, only two terms are known in this situation: for a symbol with a jump discontinuity, the leading order term in \eqref{eq:trasymptsmoothboundary} remains unaffected, whereas the sub-leading term gets enhanced to order $\log(L)L^{d-1}$. The one-dimensional case is covered in \cite{LandauWidom, Widom1981} and the works \cite{Sobolev2010, Sobolev2013, Sobolev2016, Sobolev2017} provide the extension to any dimension. An interdisciplinary interest in \eqref{eq:trhAOmegaL} originates in its relation to the bipartite entanglement entropy for a non-interacting Fermi gas, see \cite{GioevKlich2006, Helling2009, LeschkeSobolevSpitzer2014, LeschkeSobolevSpitzer2016}. In this context, recent literature contains asymptotic formulae for a generalised version of the trace \eqref{eq:trhAOmegaL}: the operator $A$ is replaced by $a(H)$ where $H=-\Delta +V$ is a Schrödinger operator with a real-valued potential $V$ and $a:\mathbb{R}\to\mathbb{R}$ is a bounded function, for instance a step function. Here, the focus lies on (random) ergodic potentials in \cite{PasturSlavin2014, KirschPastur2014, ElgartPastur2016, Dietlein2018} and periodic potentials in \cite{PfirschSobolev2018}.
\vspace{0.3cm}

We are interested in the asymptotic behaviour of \eqref{eq:trhAOmegaL} for smooth symbols $\sigma$ but for a set $\Omega$ with non-smooth boundary. As before, assume that $\sigma\in\mathcal{S}(\mathbb{R}^d)$ and that $h$ is an entire function with $h(0)=0$. In \cite{Widom1960} the author dealt with polytopes $\Omega$ and proved a two-term asymptotic expansion of the trace \eqref{eq:trhAOmegaL}. Recently, this result was extended to a larger class of domains, see \cite{Sobolev2018}. Namely, let $\Omega$ be a bounded Lipschitz region with piecewise $C^1$-boundary. Then \cite{Sobolev2018} contains the asymptotics
\begin{align}\label{eq:trhAOmegaLLipschitz}
\tr h(A_{\Omega_L})=L^d \mathcal{B}_d + L^{d-1}\mathcal{B}_{d-1} + o(L^{d-1}),
\end{align} 
as $L\to\infty$, where the coefficients $\mathcal{B}_j=\mathcal{B}_j(\Omega,h,\sigma)$, $j=d,d-1$, are given via the same formulas as in the smooth boundary case. The coefficient $\mathcal{B}_{d}$ agrees with \eqref{eq:Bd} and a formula for $\mathcal{B}_{d-1}$ can be found, for instance, in \cite[Thm. 1.1]{Roccaforte1984}. In particular, one observes that the edges (or if $d=2$ the corners) of $\Omega$ do not enter the trace asymptotics up to order $L^{d-1}$. In the special case of cubes $\Omega$, \cite[Thm. 2.2]{Dietlein2018} actually implies complete asymptotics for \eqref{eq:trhAOmegaL}, consisting of $d+1$ terms. However, the latter result is established in the more general framework of $\mathbb{Z}^d$-ergodic operators. This entails an exclusively abstract formulation of the asymptotic coefficients, which makes it difficult to relate them to the smooth boundary case. In addition, \cite[Thm. 2.2]{Dietlein2018} makes for the Wiener-Hopf case unnecessary symmetry assumptions; for instance, it is applicable to radially symmetric symbols $\sigma$. 

Similar results have been obtained in the discrete setting, where $A_{\Omega_L}$ is replaced by the doubly-infinite $d$-dimensional Toeplitz matrix $T$ restricted to a scaled lattice subset $\Lambda_L\subset \mathbb{Z}^d$. For polytopes $\Lambda$, the work \cite{Doktorskii1984} provides a two-term asymptotic formula for $\tr h(T_{\Lambda_L})$, analogous to the result in \cite{Widom1960}. When $\Lambda$ is a cuboid, the authors of \cite{Seghier1986} and \cite{Thorsen1996} proved a $(d+1)$-term asymptotic formula for $\tr h(T_{\Lambda_L})$, under the additional assumption that the symbol of the Toeplitz matrix allows a specific factorisation. In \cite{Kateb2000} these results were recovered and further insights were given on the inverses of Toeplitz matrices on convex polytopes. Moreover, the recent work \cite{Rinkel2014} treats triangles $\Lambda\subset \mathbb{Z}^2$ and provides a two-term asymptotic formula for $\tr T_\Lambda^{-1}$ with a new formula for the sub-leading coefficient.

In this paper, our objective is to investigate further the term of order $L^{d-2}$ in \eqref{eq:trhAOmegaLLipschitz}. We restrict ourselves to dimension two and deal with the case that $\Omega=P\subset\mathbb{R}^2$ is the interior of a polygon. By the latter we mean that $P$ is bounded and $\partial P$ is the finite disjoint union of piecewise linear, closed  curves; we do not require that $P$ be (simply) connected or convex. In particular, and in contrast to all previous works on the complete asymptotics of \eqref{eq:trhAOmegaL}, we shall deal with corners of any angle. With $h$ as above and slightly relaxed assumptions on $\sigma$ we obtain complete asymptotics for $\tr h(A_{P_L})$, consisting of three terms, see Theorem \ref{thm:abstractasympt}. More precisely, we provide constants $c_j=c_j(P,h,\sigma)$ such that
\begin{align}\label{eq:asympttrpolygonintro}
\tr h(A_{P_L})=L^2c_2+Lc_1+c_0+\mathcal{O}(L^{-\infty}),
\end{align}
as $L\to\infty$. As it can be inferred from formula \eqref{eq:trhAOmegaLLipschitz}, the coefficient $c_2$ incorporates the area of the polygon $P$ and $c_1$ depends on the lengths of its edges and their directions. However, our main focus lies on the constant order coefficient $c_0$, which contains contributions from each corner of the polygon. In Theorem~\ref{thm:abstractasympt} we provide a formula for $c_0$ given in terms of abstract traces, similarly to \cite[Thm. 2.2]{Dietlein2018}. Yet, in the polygon case $c_0$ includes additional terms due to the presence of non-parallel edges. Furthermore, we compute $c_0$ explicitly as a function of the polygon's interior angles for radially symmetric symbols $\sigma$ and quadratic test functions $h$, see Theorem \ref{thm:rotsymmasymptsquare}. As a consequence, one can compare $c_0$ with the corresponding coefficient in the smooth boundary case and we obtain the following result: for a two-dimensional domain $\Omega$, one can determine from the constant order term of the trace asymptotics \eqref{eq:trhAOmegaLLipschitz} whether $\Omega$ has a smooth boundary or it is a polygon, see Corollary \ref{corollary}. In addition, the coefficient $c_0$ for the polygon $P$ can not be obtained from \eqref{eq:trasymptsmoothboundary} via approximation of $P$ by domains with smooth boundary. This anomaly resembles the analogous result for the constant order term in the heat trace asymptotics for the Dirichlet Laplacian on a two-dimensional domain with corners, see~\cite{MazzeoRowlett2015}.

A few remarks on the structure of the paper are in order. We start by formulating our main results: Theorems \ref{thm:abstractasympt} and \ref{thm:coefficients} state the asymptotics \eqref{eq:asympttrpolygonintro} with various formulae for the coefficients $c_j$ and Theorem \ref{thm:rotsymmasymptsquare} deals with the radially symmetric case. The trace norm estimates that enter the proofs of Theorems~\ref{thm:abstractasympt} and \ref{thm:coefficients} are collected in Section~\ref{sec:trnormestimates}. In Section~\ref{sec:redasymptcorners} we apply these trace norm bounds to extract the leading order term of the asymptotics \eqref{eq:asympttrpolygonintro}. Moreover, we reduce the remaining part to individual corner contributions, which only depend on the corner angle and the lengths and directions of the enclosing edges. The trace asymptotics corresponding to a single corner of the polygon are provided in Section \ref{sec:sectopslocaltrasympt}, which completes the proof of Theorem~\ref{thm:abstractasympt}. The proofs of Theorems~\ref{thm:coefficients} and \ref{thm:rotsymmasymptsquare} can be found in Sections \ref{sec:evalcoeff} and \ref{sec:rotsymmsymbols}.   

To conclude the introduction, we fix some general notation that will be applied throughout the paper. 
If $f,g$ are non-negative functions, we write $f\lesssim g$ or $g\gtrsim f$ if $f\leq Cg$ for some constant $C>0$. This constant will always be independent of the scaling parameter $L$, but it might depend on the test function $h$, the symbol $\sigma$, and the geometry of the polygon $P$. We will comment on its explicit dependence whenever necessary. For $x\in\mathbb{R}^d$, we use the notation $\ang{x}:=(1+|x|^2)^{1/2}$, where $|\cdot|$ is the standard Euclidean norm. Moreover, $Q_x$ denotes the (closed) unit cube centred at $x$ and $B_r(x)$ is the (closed) unit ball of radius $r>0$ around $x$ (with respect to $|\cdot|$).

\textit{Acknowledgement.} The author is very grateful to Adrian Dietlein and Alexander V. Sobolev for illuminating discussions and valuable comments on the manuscript.

\section{Results}\label{sec:results}

Let $h:\mathbb{C}\to\mathbb{C}$ be an entire function with $h(0)=0$ and consider a symbol $\sigma\in\mathsf{W}^{\infty,1}(\mathbb{R}^2)$, see \eqref{eq:defWinfty1} for the definition. These assumptions will be sufficient to obtain the asymptotic trace formula \eqref{eq:asympttrpolygonintro} with well-defined coefficients $c_j=c_j(P,h,\sigma)$. In order to write out the formulas for the coefficients we need to fix some notation for the polygon $P$. 

\subsection{Notation for the polygon $\boldsymbol{P}$ and coefficients in the asymptotics}
\label{subsec:notationPandcoeff}
Let $\Xi(P)\subset \mathbb{R}^2$ denote the set of vertices of $P$ and  $\mathcal{E}(P)$ the set of edges of $P$. In the following we specify the contribution of each edge $E\in\mathcal{E}(P)$ and each corner at $X\in\Xi(P)$ to the asymptotics \eqref{eq:asympttrpolygonintro}.

First, fix an edge $E\in\mathcal{E}(P)$. Let $\nu_E$ be its inward pointing unit normal vector and let $\tau_E$ be the unit tangent vector such that the frame $(\tau_E,\nu_E)$ has the standard orientation in $\mathbb{R}^2$. This induces an orientation on $\partial P$. Introduce the half-space
\begin{align}\label{eq:defHE}
H_E:=\lbrace y\in\mathbb{R}^2: y\cdot \nu_E\geq 0\rbrace,
\end{align}
and the semi-infinite strip of unit width,
\begin{align}\label{eq:defSE}
S_E:=\lbrace a\tau_E + b\nu_E: (a,b)\in [0,1]\times [0,\infty)\rbrace\subset H_E.
\end{align}
We also label the interior angles between $E$ and its adjacent edges by $\gamma_E^{(1)}$ and $\gamma_E^{(2)}$. For definiteness the enumeration is chosen in accordance with the orientation of $\partial P$. However, the latter is not of much relevance as we will mainly be interested in a symmetric function of the angles, $F:\mathcal{E}(P)\to\mathbb{R}$,
\begin{align}\label{eq:defF}
F(E):=-\cot(\gamma_E^{(1)})-\cot(\gamma_E^{(2)}).
\end{align}
Note that $F(E)=0$ if and only if $\gamma_E^{(1)}+\gamma_E^{(2)}\in \lbrace \pi, 2\pi, 3\pi\rbrace$, i.e.~if and only if the edges adjacent to $E$ are parallel.
Defining also the function 
\begin{align}\label{eq:defh1}
h_1(z):=h(z)-zh'(0),
\end{align}
we introduce the following coefficients corresponding to the edge $E$, which are finite under our assumptions on $h$ and $\sigma$, see also Theorem \ref{thm:abstractasympt}. We set
\begin{align}\label{eq:defa1nuE}
a_1(\nu_E):=\tr\big(\chi_{S_E}\big[h_1(A_{H_E})-h_1(A)\big]\big),
\end{align}
with $S_E$ and $H_E$ as in \eqref{eq:defHE}, \eqref{eq:defSE}. Note that the strip $S_E$ on the right-hand side of \eqref{eq:defa1nuE} may actually be shifted along the edge $E$, leaving the value of $a_1(\nu_E)$ unchanged since the operator $h_1(A_{H_E})-h_1(A)$ is translation-invariant in the direction $\tau_E$. Similarly, we define the coefficient
\begin{align}\label{eq:defa0nuE}
a_0(\nu_E):=\tr\big(\chi_{S_E}M(x\cdot\nu_E)\big[h_1(A_{H_E})-h_1(A)\big]\big),
\end{align} 
where $M(x\cdot\nu_E)$ is the multiplication operator
\begin{align}
[M(x\cdot\nu_E)f](x):=(x\cdot\nu_E) f(x),
\end{align}
for any function $f:\mathbb{R}\to\mathbb{C}$. Clearly, also the operator $M(x\cdot\nu_E)\big[h_1(A_{H_E})-h_1(A)\big]$ is invariant with respect to translations along the edge $E$.

Fix now a vertex $X\in\Xi(P)$. Its adjacent edges are named $E^{(1)}(X)$ and $E^{(2)}(X)$, where the enumeration is again chosen according to the orientation of $\partial P$. Corresponding to the vertex $X$ we have the two half-spaces
\begin{align}\label{eq:defHjX}
H^{(j)}(X):=H_{E^{(j)}(X)}, \ j=1,2,
\end{align}
compare with \eqref{eq:defHE}. Moreover, let $\gamma_X\in(0,\pi)\cup(\pi,2 \pi)$ denote the interior angle at $X$. In the following, we distinguish convex and concave corners of the polygon, employing the notation
\begin{align*}
\Xi_{\lessgtr}(P):=\lbrace X\in \Xi(P): \gamma_X\lessgtr \pi\rbrace.
\end{align*}
Define the semi-infinite sector modelling the corner at $X\in\Xi(P)$ by 
\begin{align}\label{eq:defC(X)}
C(X):=\begin{cases}
H^{(1)}(X)\cap H^{(2)}(X),&\ X\in\Xi_<(P),\\
H^{(1)}(X)\cup H^{(2)}(X),&\ X\in\Xi_>(P).
\end{cases}
\end{align}
If $X\in\Xi_<(P)$, the corner at $X\in\Xi(P)$ or equivalently the sector $C(X)$ is \textit{convex}, otherwise we call it \textit{concave}. We are now ready to introduce coefficients corresponding to vertices $X\in\Xi(P)$. \\[2ex]
If $X\in\Xi_<(P)$, we define
\begin{align}
\label{eq:b0Xconvex}
b_0(X)&:=\tr\big(\chi_{C(X)} \big[h_1(A_{C(X)})-h_1(A_{H^{(1)}(X)})-h_1(A_{H^{(2)}(X)})+h_1(A)\big]\big),
\end{align}
with $C(X)$ and $H^{(j)}(X)$, $j=1,2$, defined in \eqref{eq:defC(X)} and \eqref{eq:defHjX}, respectively.\\[2ex]
If $X\in\Xi_>(P)$, we set
\begin{align}\label{eq:b0Xconcave}
b_0(X)&:=\tr\big(\chi_{H^{(1)}(X)\cap H^{(2)}(X)}\big[h_1(A_{C(X)})-h_1(A)\big]\big)+\tr\big(\chi_{C(X)\setminus H^{(1)}(X)}\big[h_1(A_{C(X)})-h_1(A_{H^{(2)}(X)}\big]\big)\nonumber\\[2ex]
& \hspace{2cm} +\tr\big(\chi_{C(X)\setminus H^{(2)}(X)}\big[h_1(A_{C(X)})-h_1(A_{H^{(1)}(X)}\big]\big).
\end{align}
\subsection{Main result} Our first and main theorem provides a complete asymptotic expansion of $\tr h(A_{P_L})$ and contains formulas for all the coefficients in \eqref{eq:asympttrpolygonintro}.
\begin{theorem}
\label{thm:abstractasympt}
Assume that $\sigma\in \mathsf{W}^{\infty,1}(\mathbb{R}^2)$, see \eqref{eq:defWinfty1}, and let $h:\mathbb{C}\to\mathbb{C}$ be an entire function with $h(0)=0$. Then we have the asymptotic formula
\begin{align}\label{eq:abstractthmasymptformula}
\tr h(A_{P_L})=L^2 c_2 +L c_1 + c_0+\mathcal{O}(L^{-\infty}), 
\end{align}
as $L\to\infty$, with coefficients
\begin{align*}
c_2&=\frac{|P|}{4\pi^2}\int\limits_{\mathbb{R}^2} d\xi\, (h\circ \sigma) (\xi)\\
c_1&=\sum\limits_{E\in \mathcal{E}(P)}|E|\,a_1(\nu_E), \\
c_0&=\sum\limits_{E\in\mathcal{E}(P)}F(E)\,a_0(\nu_E) + \sum\limits_{X\in\Xi(P)}b_0(X).
\end{align*}
In particular, for all $E\in\mathcal{E}(P)$ and $X\in\Xi(P)$, the coefficients $a_1(\nu_E)$, $a_0(\nu_E)$, and $b_0(X)$  are well-defined, see Subsection \ref{subsec:notationPandcoeff} for their definition.
\end{theorem}
\begin{remark}
\begin{enumerate}
\item As we know from the formula \eqref{eq:trhAOmegaLLipschitz}, the corners of the polygon do not affect the two leading coefficients in the asymptotics compared to the smooth boundary case. However, the above formula for $c_0$ shows that the corners do enter the trace asymptotics at the constant order.
\item An edge $E\in\mathcal{E}(P)$ does not contribute to the coefficient $c_0$ if $F(E)=0$, i.e.~if the edges adjacent to $E$ are parallel, see also \eqref{eq:defF}. In particular, all contributions from the edges to $c_0$ vanish if, for instance, $P$ is a parallelogram. As it becomes clear from the proof of the theorem, the edge contributions to $c_0$ are in fact aggregated local contributions from corners of $P$.
\item We emphasise that the coefficients $b_0(X)$ are defined by the two distinct formulas \eqref{eq:b0Xconvex} and \eqref{eq:b0Xconcave}, depending on the type of the corner at $X\in\Xi(P)$. 
\end{enumerate}
\end{remark}
The coefficients $a_1(\nu_E)$ and $a_0(\nu_E)$, which only depend on the half-space operators $h(A_{H_E})$ and the full-space operator $h(A)$, may be rewritten in terms of one-dimensional Wiener-Hopf operators. This is the content of the next theorem. Here, we use the popular notation
\begin{align*}
W(\sigma):=A_{[0,\infty)}(\sigma),
\end{align*}
for $\sigma\in\mathsf{L}^\infty(\mathbb{R})$. As anticipated, the formula \eqref{eq:c1asWienerHopf} for $c_1$ reduces the corresponding formula from the smooth boundary case, see \cite[Thm.]{WIDOM1980}. 
\begin{theorem}\label{thm:coefficients}
Let $\sigma\in\mathsf{W}^{\infty,1}(\mathbb{R}^2)$ and let $h:\mathbb{C}\to\mathbb{C}$ be an entire function with $h(0)=0$. Define, for $E\in \mathcal{E}(P)$ and $t\in \mathbb{R}$, the family of one-dimensional symbols
\begin{align}\label{eq:defsigmaEt}
\mathbb{R}\ni\xi \mapsto\sigma_{E,t}(\xi):=\sigma(t\tau_E+\xi\nu_E).
\end{align}
Then, for all $E\in\mathcal{E}(P)$, the coefficients $a_1(\nu_E)$ and $a_0(\nu_E)$ in Theorem \ref{thm:abstractasympt} may be rewritten as
\begin{align}\label{eq:a1nuEasWienerHopf}
a_1(\nu_E)&=\frac{1}{2\pi}\int\limits_\mathbb{R}dt\, \tr\big[h\big\lbrace W(\sigma_{E,t})\big\rbrace-W(h\circ\sigma_{E,t})\big],\\
\label{eq:a0nuEasWienerHopf}
a_0(\nu_E)&=\frac{1}{2\pi}\int\limits_{\mathbb{R}}dt\, \tr\big(M(x)\big[h\lbrace W(\sigma_{E,t})\rbrace-W(h\circ\sigma_{E,t})\big]\big),
\end{align}
where $M(x)$ denotes multiplication by $x$ on $L^2(\mathbb{R})$. In particular, we have that
\begin{align}\label{eq:c1asWienerHopf}
c_1=\sum\limits_{E\in\mathcal{E}(P)}\frac{|E|}{2\pi}\int\limits_\mathbb{R}dt\, \tr\big[h\big\lbrace W(\sigma_{E,t})\big\rbrace-W(h\circ\sigma_{E,t})\big].
\end{align}
\end{theorem}

\begin{remark}
The advantage of formulas \eqref{eq:a1nuEasWienerHopf} and \eqref{eq:a0nuEasWienerHopf} lies in the fact that explicit formulas for the traces of one-dimensional Wiener-Hopf operators are known. Assuming for simplicity that $\sigma\in\mathcal{S}(\mathbb{R}^2)$, \cite[Prop. 5.4]{Widom1985} implies that
\begin{align*}
a_1(\nu_E)=\frac{1}{8\pi^3}\int\limits_\mathbb{R}dt\int\limits_{\mathbb{R}} d\xi_1\int\limits_{\mathbb{R}} d\xi_2\,\frac{h(\sigma_{E,t}(\xi_1))-h(\sigma_{E,t}(\xi_2))}{\sigma_{E,t}(\xi_1)-\sigma_{E,t}(\xi_2)}\frac{\sigma_{E,t}'(\xi_2)}{\xi_2-\xi_1},
\end{align*}        
where the integral over $\xi_2$ is interpreted as a Cauchy principal value. Referring to the same proposition, one similarly gets that
\begin{align*}
a_0(\nu_E)=&-\frac{1}{64\pi^2}\int\limits_\mathbb{R}dt\int\limits_{\mathbb{R}} d\xi \, h''(\sigma_{E,t}(\xi))\sigma'_{E,t}(\xi)^2 \\
&- \frac{1}{32\pi^4}\int\limits_\mathbb{R}dt\int\limits_{\mathbb{R}} d\xi_1\int\limits_{\mathbb{R}} d\xi_2\int\limits_{\mathbb{R}} d\xi_3\,\Bigg\lbrace \sum\limits_{k=1}^3\frac{h(\sigma(\xi_k))}{\prod\limits_{j\neq k}[\sigma(\xi_k)-\sigma(\xi_j)]}\Bigg\rbrace\frac{\sigma_{E,t}'(\xi_2)}{\xi_2-\xi_1}\frac{\sigma_{E,t}'(\xi_3)}{\xi_3-\xi_1}.
\end{align*}
\end{remark}
\subsection{The radially symmetric case}
In contrast to the above, the coefficients $b_0(X)$, see \eqref{eq:b0Xconvex} and \eqref{eq:b0Xconcave}, can naturally not be transformed into integrals over traces of one-dimensional fibre operators since they incorporate the truly two-dimensional sector operators $h(A_{C(X)})$. This makes their explicit calculation rather involved. However, we manage to compute the coefficients $b_0(X)$ in the special case when $h$ is a quadratic polynomial and the symbol $\sigma$ is radially symmetric. By the latter we mean that, for any orthogonal matrix $O\in \mathbb{R}^{2\times 2}$ and for all $\xi\in\mathbb{R}^2$,
\begin{align*}
\sigma(\xi)=\sigma(O\xi).
\end{align*}
Define
\begin{align}\label{eq:defcheckf}
\check{f}(x):=(2\pi)^{-d/2}(\mathcal{F}^\ast f)(x)=(2\pi)^{-d}\int\limits_{\mathbb{R}^d}d\xi\, e^{i\xi\cdot x}f(\xi),
\end{align}
for functions $f\in\mathsf{L}^1(\mathbb{R}^d)$, so that the operator $A$ has the difference kernel
\begin{align*}
A(x,y)=\check{\sigma}(x-y), \ x,y\in\mathbb{R}^d.
\end{align*}
If $\sigma$ is radially symmetric, so is $\check{\sigma}$ and we shall write, slightly abusing notation,
\begin{align*}
\sigma(|\xi|)=\sigma(\xi), \ \check{\sigma}(|x|)=\check{\sigma}(x),
\end{align*}
for all $x,\xi\in\mathbb{R}^2$. In the following theorem all coefficients $c_j$ in the asymptotics \eqref{eq:asympttrpolygonintro} are computed explicitly for such symbols $\sigma$ and quadratic test functions $h$. Again, our focus lies on the coefficient $c_0$ since the formulas for $c_2$ and $c_1$ are known to be the same as in the smooth boundary case.
\begin{theorem}\label{thm:rotsymmasymptsquare}
Suppose that $\sigma\in\mathsf{W}^{\infty,1}(\mathbb{R}^2)$ is radially symmetric and let $h(z)=z^2+bz$ for some $b\in\mathbb{C}$. Then we have that
\begin{align*}
\tr h(A_{P_L})=L^2c_2+Lc_1+c_0+\mathcal{O}(L^{-\infty}), 
\end{align*} 
as $L\to\infty$, with
\begin{align*}
c_2&=\frac{|P|}{2\pi}\int\limits_0^\infty dR\, R\,(h\circ \sigma)(R),\\
c_1&=-2\left\vert{\partial P}\right\vert\int\limits_0^\infty dr\,r^2\check{\sigma}(r)^2,\\
c_0&= \sum\limits_{X\in \Xi(P)} \tfrac{1}{2}\big[1+(\pi-\gamma_X)\cot\gamma_X\big]\int\limits_0^\infty dr\, r^3 \check{\sigma}(r)^2. 
\end{align*}
\end{theorem}
\begin{remark}
\begin{enumerate}
\item As in Theorem \ref{thm:abstractasympt}, the coefficients $c_1$ and $c_0$ only depend on the test function $h$ via the function $h_1(z)=z^2$.
\item Notice that, due to the radial symmetry of $\sigma$, the dependence of the coefficients $c_j$, $j=0,1,2$, on the geometry of $P$ separates from their dependence on the symbol $\sigma$.
\item Interestingly, the contribution of convex corners and concave corners to $c_0$ are obtained via the same formula, in contrast to the two distinct formulas \eqref{eq:b0Xconvex}, \eqref{eq:b0Xconcave}. 
\end{enumerate}
\end{remark}

The explicit formula for the coefficient $c_0$ given in Theorem \ref{thm:rotsymmasymptsquare} allows us to compare it with the corresponding coefficient $\mathcal{B}_0$ from the smooth boundary case, see \eqref{eq:trasymptsmoothboundary}. As in the theorem let $h$ be a quadratic test function and assume that $\sigma\in\mathsf{W}^{\infty,1}(\mathbb{R}^2)$ is radially symmetric.  Applying \cite[Thm. 1.1]{Roccaforte1984}, one gets that, for any bounded $\Omega\subset\mathbb{R}^2$ with smooth boundary, 
\begin{align}\label{eq:B0vanishes}
\mathcal{B}_0=\mathcal{B}_0(\Omega,h,\sigma)=0.
\end{align}
To our knowledge, this surprising fact has not been noted explicitly before and it even holds without the radial symmetry of $\sigma$. For the reader's convenience we provide a proof of \eqref{eq:B0vanishes} in an appendix to this paper, see Lemma \ref{lem:appendix}. In contrast to the above, the function
\begin{align}\label{eq:defanglefunction}
f(\gamma):=1+(\pi-\gamma)\cot(\gamma)
\end{align}
is positive on $(0,\pi)\cup(\pi,2\pi)$. This yields the following corollary.  
\begin{corollary} \label{corollary} Let $h(z)=z^2+bz$ and suppose that the symbol $0\neq \sigma\in\mathsf{W}^{\infty,1}(\mathbb{R}^2)$ is real-valued and radially symmetric. Moreover, assume that $P$ is a polygon and $\Omega\subset\mathbb{R}^2$ is a bounded set with smooth boundary. Then one has that
\begin{align*}
c_0(P,h,\sigma)>0,
\end{align*}
while
\begin{align*}
\mathcal{B}_0(\Omega,h,\sigma)=0,
\end{align*}
where $c_0$ and $\mathcal{B}_0$ are the constant order coefficients from \eqref{eq:asympttrpolygonintro} and \eqref{eq:trasymptsmoothboundary}.
\end{corollary}
\begin{remark}
The corollary implies the following: consider a bounded set $\Lambda\subset\mathbb{R}^2$ with either smooth or piecewise linear boundary. Then the type of the boundary can be determined from the spectral asymptotics of $A_{\Lambda_L}$, as $L\to\infty$.
\end{remark}
As a consequence of Corollary \ref{corollary} the constant order coefficient in the trace asymptotics exhibits an anomaly, similarly to the heat trace asymptotics for the Dirichlet Laplacian on two-dimensional domains with corners, see e.g. \cite{MazzeoRowlett2015, LuRowlett2016}. Any approximation of a polygon $P$ by a sequence of smooth domains $\lbrace \Omega_n \rbrace$ can not recover the coefficient $c_0$: for functions $h$ and $\sigma$ as in the corollary, one gets that
\begin{align*}
\mathcal{B}_0(\Omega_n,h,\sigma)=0\nrightarrow c_0(P,h,\sigma),
\end{align*}
as $n\to\infty$. On the other hand, the approximation of domains with smooth boundary by polygons works fine. As a simple but representative example consider a disc $\Omega$ and let $\lbrace P_n \rbrace$ be a sequence of inscribed regular $n$-gons, approximating $\Omega$. As the function $f$, see \eqref{eq:defanglefunction}, vanishes to second order at $\gamma=\pi$, one easily checks that
\begin{align*}
c_0(P_n,h,\sigma)\to 0=\mathcal{B}_0(\Omega,h,\sigma),
\end{align*}
as $n\to\infty$.

We also point out that one may apply Theorem~\ref{thm:rotsymmasymptsquare} to compute the particle number fluctuation (PNF) of a free Fermi gas at positive temperature with respect to the spatial bipartition ${\mathbb{R}^2=P_L\,\dot{\cup}\, \mathbb{R}^2\setminus P_L}$. Namely, the PNF is given by
\begin{align*}
\tr h(A_{P_L}),
\end{align*}
with 
\begin{align*}
h(x)=x(1-x), \ \sigma(\xi)=\big[1+\exp\big(\frac{\xi^2-\mu}{T}\big)\big]^{-1},
\end{align*}
see \cite{Klich2006}. Here, $\mu\in\mathbb{R}$ is the chemical potential and $T>0$ denotes temperature. Corollary \ref{corollary} allows us to compare the PNF for a scaled polygon $P_L$ with the PNF for a scaled set $\Omega_L$ with smooth boundary: if $P$ and $\Omega$ have the same area and perimeter, then the PNF for the polygon $P_L$ is strictly larger than the PNF for $\Omega_L$, as $L\to\infty$. 

\subsection{Strategy of the proofs}\label{subsec:strategyproofs}
Let us comment on the basic ideas for the proofs of Theorems~\ref{thm:abstractasympt}, \ref{thm:coefficients}, and \ref{thm:rotsymmasymptsquare}. 

The strategy of the proof of Theorem \ref{thm:abstractasympt} is as follows. The leading order term in the asymptotics originates from approximating the operator $h(A_{P_L})$ by its bulk approximation $\chi_{P_L} h(A)\chi_{P_L}$, which is a very familiar idea. Indeed, one easily computes that
\begin{align}\label{eq:bulkapproxtrace}
\tr\big(\chi_{P_L} h(A)\chi_{P_L}\big)=\int\limits_{P_L}dx\, h(A)(x,x)=|P_L|(h\circ \sigma\widecheck{)\,}\!(0) = L^2c_2.
\end{align}
Subtracting the latter from $\tr h(A_{P_L})$ leaves a remainder that is independent of the linear part of $h$, hence we may replace $h$ by the function $h_1(z)=h(z)-zh'(0)$:
\begin{align}\label{eq:diffleadingsubleading}
\tr\big(\chi_{P_L}\big[h(A_{P_L})-h(A)\big]\chi_{P_L}\big)=\tr\big(\chi_{P_L}\big[h_1(A_{P_L})-h_1(A)\big]\chi_{P_L}\big).
\end{align}
For the following steps we mainly rely on the locality of the operator $A$: due to the assumptions on the symbol $\sigma$, the kernel $A(x,y)=\check{\sigma}(x-y)$ decays super-polynomially away from the diagonal, see Lemma \ref{lem:sigma}. As a first consequence, we can prove that the operator $\chi_{P_L}\big[h_1(A_{P_L})-h_1(A)\big]\chi_{P_L}$ is concentrated on the boundary $\partial P_L$. More precisely, defining for small but fixed $\epsilon>0$ the (unscaled) one-sided $\epsilon$-neighbourhood of $\partial P$,
\begin{align}\label{eq:defV}
\mathcal{V}:=\mathcal{V}^{(\epsilon)}:=\lbrace y\in P\cup\partial P: \dist(y,\partial P)\leq \epsilon\rbrace,
\end{align}
we show that
\begin{align}
\tr\big(\chi_{P_L}\big[h_1(A_{P_L})-h_1(A)\big]\big)=\tr\big(\chi_{\mathcal{V}_L}\big[h_1(A_{P_L})-h_1(A)\big]\big)+\mathcal{O}(L^{-\infty}),
\end{align}
as $L\to\infty$. It is convenient to partition $\mathcal{V}$ into corner neighbourhoods $\mathcal{N}(X)$, $X\in\Xi(P)$, that extend along half of the edges $E^{(1)}(X)$ and $E^{(2)}(X)$, see Figure~\ref{fig:1} below. 
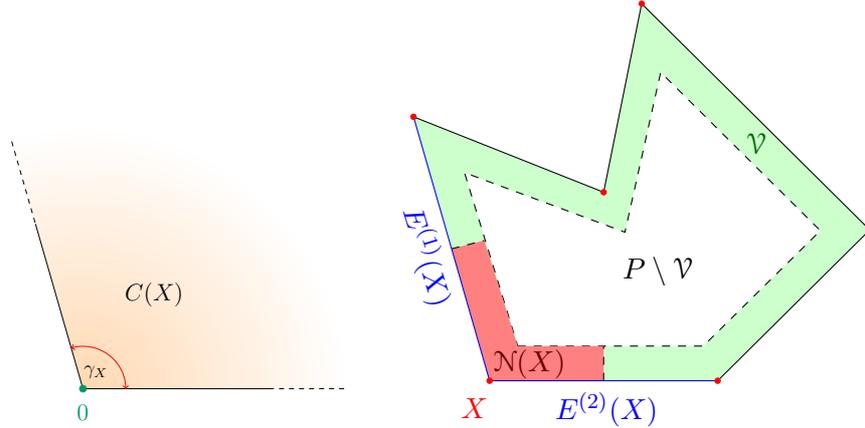
\begin{figure}[!t]
\caption{The sector $C(X)$, the one-sided boundary neighbourhood $\mathcal{V}$, and the corner neighbourhood $\mathcal{N}(X)$}\label{fig:1}
\resizebox{5cm}{!}{
\begin{tikzpicture}
\clip (-1.75,-1) rectangle (6,6);
      \shade[shading=radial, inner color=orange!30!white, outer color= white] (0,0) circle (5.5cm);
      \fill[color=white] (0,0) -- (5.5cm,0cm) arc (0:-254:5.5cm);
    \draw
		(-1,3.5) coordinate (a) -- (0,0)
		coordinate (b) -- 
		(4,0) coordinate (c)
		pic["\large $\gamma_X$", draw=red, <->, angle 
		eccentricity=0.5, angle radius=0.9cm]
		{angle=c--b--a};
	\draw[dashed] 
		(-1.5,5.25) -- (-1,3.5); 
	\draw[dashed]		
		(5.5,0) -- (4,0);
	\filldraw[green!60!blue] (0,0) circle(2pt);
	\draw (0,-0.2) node[anchor=north, text=green!60!blue]{\Large $0$};
	\node at (1.5,2){\Large $C(X)$};
\end{tikzpicture}
}
\begin{tikzpicture}
	\fill[green!20!white]
		(0,0) -- (3,0) -- 
		(5,2) -- (2,5) -- (1.5,2.5) --
		(-1,3.5) -- cycle;
	\fill[white]
		(0.37,0.46) -- (2.81,0.46) -- (4.35,1.99) -- 
		(2.25,4.09) -- (1.78,1.97) -- (-0.33,2.76) -- 
		cycle;	
	\draw 
		(3,0) -- 
		(5,2) -- (2,5) -- (1.5,2.5) --
		(-1,3.5) ;
	\draw[dashed]
		(0.37,0.46) -- (2.81,0.46) -- (4.35,1.99) -- 
		(2.25,4.09) -- (1.78,1.97) -- (-0.33,2.76) -- 
		cycle;
	\node at (2.21,1.45) {$P\setminus\mathcal{V}$};
	\fill[color=red!50!white] (0,0) -- (1.5,0) -- (1.5,0.46) -- (0.37,0.46) -- cycle;
	\fill[color=red!50!white] (0,0) -- (-0.5,1.75) -- (-0.07,1.86) -- (0.37,0.46) -- cycle;
	\draw[dashed] (1.5,0) -- (1.5,0.46) (-0.5,1.75) -- (-0.07,1.86);
	\draw[blue] (-1,3.5)--(0,0) -- (3,0);
\foreach \position in {(0,0),(3,0),(5,2),(2,5),(1.5,2.5),(-1,3.5)} 
	\filldraw[red] \position circle (1pt);
	\node[color=green!40!black,very thick] at (3.51,3.15) { $\mathbf{\mathcal{V}}$};
	\node[color=red!40!black,very thick] at (0.53,0.22) {$\mathcal{N}(X)$};
	\draw[red] (-0.2,-0.1) node[anchor=north]{$X$};
\node[color=blue] at (1.54,-0.38) {$E^{(2)}(X)$};
\node[color=blue, rotate=-73] at (-0.82,1.64) {$E^{(1)}(X)$};
\end{tikzpicture}
\end{figure}
This reduces the problem to computing the asymptotics of 
\begin{align}\label{eq:cornertrace}
\tr\big(\chi_{\mathcal{N}_L(X)}\big[h_1(A_{P_L})-h_1(A)\big]\big),
\end{align}
for a fixed vertex $X\in\Xi(P)$. In view of the translation-invariance of $A$ we may assume that $X=0$, hence the sector $C(X)$ models the corner at $X\in\Xi(P)$, see \eqref{eq:defC(X)} and Figure \ref{fig:1}. Again the locality of the operator $A$ implies that one can replace the operator $h_1(A_{P_L})$ in \eqref{eq:cornertrace} by the $L$-independent sector operator $h_1(A_{C(X)})$:
\begin{align}\label{eq:replacebysector}
\tr\big(\chi_{\mathcal{N}_L(X)}\big[h_1(A_{P_L})-h_1(A)\big]\big)=\tr\big(\chi_{\mathcal{N}_L(X)}\big[h_1(A_{C(X)})-h_1(A)\big]\big) +\mathcal{O}(L^{-\infty}).
\end{align}
Thus, we have completely localised the problem to the corner at $X\in\Xi(P)$. It remains to prove that the right-hand side of \eqref{eq:replacebysector} exhibits a two-term asymptotic expansion with super-polynomial error: the leading order term, linear in $L$, results from the parts of $\mathcal{N}(X)$ near an edge $E^{(1)}(X)$ or $E^{(2)}(X)$, whereas its constant order correction is solely produced by the fraction of $\mathcal{N}(X)$ close to the vertex $X$. In order to extract these two terms we provide a trace-class regularisation of the operator $h_1(A_{C(X)})$, see Proposition \ref{prop:regularisationcornerop}. This part of the proof shows some commonalities with the analysis in \cite{Dietlein2018} for the case of cubes. Summing up the contributions from all $X\in\Xi(P)$ finishes the proof of Theorem~\ref{thm:abstractasympt}.

Theorem~\ref{thm:coefficients} is deduced from Theorem \ref{thm:abstractasympt}. Here, the key observation is that, for a fixed edge $E\in\mathcal{E}(P)$, the operator $h(A_{H_E})-h(A)$ is invariant with respect to translations along $E$. As a consequence, it is unitarily equivalent to a direct integral over one-dimensional fibre operators that are parametrised by the tangential coordinate. Not surprisingly, these fibre operators can be rewritten in terms of one-dimensional Wiener-Hopf operators, which results in the formulas \eqref{eq:a1nuEasWienerHopf} and \eqref{eq:a0nuEasWienerHopf} for the coefficients $a_1(\nu_E)$ and $a_0(\nu_E)$.

The proof of Theorem~\ref{thm:rotsymmasymptsquare} requires the evaluation of all the coefficients $c_j$, $j=0,1,2$, from Theorem~\ref{thm:abstractasympt}. To compute $a_1(\nu_E)$ and $a_0(\nu_E)$ for all $E\in\mathcal{E}(P)$ we apply Theorem~\ref{thm:coefficients}. Moreover, the specific choice of the function $h$ allows us to evaluate $b_0(X)$ for each $X\in\Xi(P)$ via a straightforward calculation. Here, the radial symmetry of the symbol $\sigma$ is essential to extract the dependence of $b_0(X)$ on the interior angle $\gamma_X$.

\section{Trace norm estimates}\label{sec:trnormestimates}

In this section we collect the trace norm estimates that will be sufficient to prove Theorems \ref{thm:abstractasympt} and \ref{thm:coefficients}. 

\subsection{Schatten-von Neumann classes} We introduce the standard notation for Schatten-von Neumann classes $\mathfrak{S}_p$ for $p>0$, see e.g. \cite{BS}, \cite{Simon}. A compact operator $T$ is an element of $\mathfrak{S}_p$ iff its singular values $\lbrace s_k(T)\rbrace_{k=1}^\infty$ are $p$-summable, i.e.
\begin{align*}
\|T\|_p^p:=\sum\limits_{k=1}^\infty s_k(T)^p <\infty.
\end{align*}
We shall often make use of Hölder's inequality 
\begin{align}\label{eq:CSHilbertSchmidt}
\|T_1T_2\|_1\leq \|T_1\|_p\|T_2\|_q,
\end{align} 
for $T_1\in\mathfrak{S}_p,T_2\in \mathfrak{S}_q$, and $p,q>0$ such that $\tfrac{1}{p}+\tfrac{1}{q}=1$. Notice also the interpolation inequality 
\begin{align}\label{eq:interpolSpoperatornorm}
\|T\|_p^p\leq \|T\|^{p-q}\|T\|_q^q,
\end{align}
which holds if $T\in\mathfrak{S}_q$, $0<q<p$.

\subsection{Finite volume truncations of the operator $\boldsymbol{A}$}

We recall the notation
\begin{align*}
A=A(\sigma)=\mathcal{F}^\ast\sigma\mathcal{F},
\end{align*}
and 
\begin{align*}
A_\Omega=A_\Omega(\sigma)=\chi_\Omega\mathcal{F}^\ast\sigma\mathcal{F}\chi_\Omega, 
\end{align*}
where $\Omega\subseteq\mathbb{R}^d$ is a measurable subset and $\sigma:\mathbb{R}^d\to \mathbb{C}$ is the symbol of the operator $A$, acting on $L^2(\mathbb{R}^d)$. The dependence of $A$ on $\sigma$ will be mostly suppressed, unless we consider the dimension-reduced symbol as in Section \ref{sec:evalcoeff}. Let us also remind the reader of the following general notation, which was introduced in the introduction: If $f,g$ are non-negative functions, we write $f\lesssim g$ or $g\gtrsim f$ if $f\leq Cg$ for some constant $C>0$. This constant will always be independent of the scaling parameter $L$, but it might depend on the test function $h$, the symbol $\sigma$, and the geometry of the polygon $P$.

The next lemma shows that, under mild assumptions on the symbol $\sigma$, the operator $A_\Omega$ is trace class if $\Omega\subset\mathbb{R}^d$ is bounded. Even though this is well-known we provide a proof for the reader's convenience. Having the application to the polygon $P$ in mind, one deduces from \eqref{eq:hAOmegatracenorm} below that
\begin{align*}
\|h(P_L)\|_1\lesssim L^2|P|,
\end{align*}
if $\sigma\in\mathsf{L}^1(\mathbb{R}^2)\cap \mathsf{L}^\infty(\mathbb{R}^2)$, and $h:\mathbb{C}\to\mathbb{C}$ is an entire function such that $h(0)=0$. Here, the implied constant depends on $h$ and $\sigma$.
\begin{lemma}\label{lem:troponboundeddomains}
Let $\sigma\in \mathsf{L}^1(\mathbb{R}^d)$ and assume that $\Omega,\Lambda\subset \mathbb{R}^d$ are bounded sets. Then one has the bound
\begin{align}\label{eq:HSestimatefortrnormAboundeddomains}
\|\chi_\Lambda A\chi_\Omega\|_1\lesssim |\Lambda|^{1/2}|\Omega|^{1/2}\|\sigma\|_{\mathsf{L}^1(\mathbb{R}^d)},
\end{align} 
with implied constant independent of $\sigma$, $\Lambda$, and $\Omega$.
\\
If in addition $\sigma\in\mathsf{L}^{\infty}(\mathbb{R}^d)$ and $h:\mathbb{C}\to\mathbb{C}$ is an entire function with $h(0)=0$, then also the estimate
\begin{align}\label{eq:hAOmegatracenorm}
\|h(A_{\Omega})\|_1\lesssim |\Omega|
\end{align}
holds, with implied constant only depending on $h$ and $\sigma$.
\end{lemma}
\begin{proof}
We start by proving the estimate \eqref{eq:HSestimatefortrnormAboundeddomains}. Without loss of generality, we may assume that $\sigma\geq 0$ since the symbol can be decomposed as $\sigma=\sigma_1-\sigma_2 + i(\sigma_3-\sigma_4)$ for suitable functions $\sigma_j\geq 0$. We have that
\begin{align*}
\chi_\Lambda A\chi_\Omega= B_1B_2,
\end{align*}
where $B_1$ and $B_2$ are the operators on $\mathsf{L}^2(\mathbb{R}^d)$ with kernels
\begin{align*}
B_1(x,\xi)&:=(2\pi)^{-d/2}\chi_\Lambda(x)e^{ix\cdot\xi}\sqrt{\sigma(\xi)} \nonumber\\
B_2(\xi,y)&:=(2\pi)^{-d/2}\sqrt{\sigma(\xi)}e^{-iy\cdot \xi}\chi_\Omega(y). 
\end{align*}
Hence, \eqref{eq:CSHilbertSchmidt} yields
\begin{align*}
\|\chi_\Lambda A\chi_\Omega\|_1\leq \|B_1\|_2\|B_2\|_2=(2\pi)^{-d}|\Lambda|^{1/2}|\Omega|^{1/2}\|\sigma\|_{\mathsf{L}^1(\mathbb{R}^d)},
\end{align*}
which proves \eqref{eq:HSestimatefortrnormAboundeddomains}. 

Let us now assume that $\sigma\in\mathsf{L}^1(\mathbb{R}^d)\cap\mathsf{L^\infty}(\mathbb{R}^d)$ and that $h$ is as in the formulation of the lemma. The boundedness of $\sigma$ implies the (uniform) operator norm bound
\begin{align*}
\|A_\Omega\|\leq\|A\|\leq \|\sigma\|_{\mathsf{L}^\infty(\mathbb{R}^d)}.
\end{align*} 
It follows that
\begin{align*}
\|h(A_\Omega)\|_1\lesssim \|A_\Omega\|_1,
\end{align*}
with implied constant depending only on $h$ and $\|\sigma\|_{\mathsf{L}^\infty(\mathbb{R}^d)}$. In view of \eqref{eq:HSestimatefortrnormAboundeddomains} this finishes the proof of the lemma.
\end{proof}

\subsection{Symbol estimates}

Introduce, for $N\geq 0$, the Sobolev spaces
\begin{align*}
\mathsf{W}^{N,1}(\mathbb{R}^d):=\lbrace f\in\mathsf{L}^1(\mathbb{R}^d): \ \partial^\alpha f\in\mathsf{L}^1(\mathbb{R}^d)\ \text{for all} \ \alpha\in\mathbb{N}_0^d, \ |\alpha|\leq N\rbrace,
\end{align*}
with corresponding norms
\begin{align*}
\|f\|_N:=\sum\limits_{|\alpha|\leq N} \|\partial^\alpha f\|_{\mathsf{L}^1(\mathbb{R}^d)}.
\end{align*}
Moreover, set
\begin{align*}
\mathsf{W}^{\infty,1}(\mathbb{R}^d):=\bigcap\limits_{N=0}^\infty \mathsf{W}^{N,1}(\mathbb{R}^d).
\end{align*}
In view of \cite[Thm. 2.31 (2)]{Demengel2012} we note that
\begin{align*}
\mathsf{W}^{\infty,1}(\mathbb{R}^d)\subset \mathsf{C}^\infty(\mathbb{R}^d),
\end{align*}
i.e.
\begin{align}\label{eq:defWinfty1}
\mathsf{W}^{\infty,1}(\mathbb{R}^d)&=\lbrace f\in\mathsf{C}^\infty(\mathbb{R}^d): \, \partial^\alpha f\in\mathsf{L}^1(\mathbb{R}^d) \ \text{for all} \ \alpha\in\mathbb{N}_0^d\rbrace.
\end{align}
The next lemma recalls the standard fact that, for a symbol $\sigma\in\mathsf{W}^{\infty,1}(\mathbb{R}^d)$, its (inverse) Fourier transform $\check{\sigma}$, see \eqref{eq:defcheckf}, decays super-polynomially at infinity. Moreover, it provides some information on the dimension-reduced symbol, which will be useful when proving Theorem~\ref{thm:coefficients}.  
\begin{lemma}\label{lem:sigma}
Let $\sigma\in \mathsf{W}^{\infty,1}(\mathbb{R}^d)$. Then the following statements hold true.
\begin{enumerate}[(i)]
\item For all $N\in\mathbb{N}_0$, one has the bound
\begin{align*}
|\check{\sigma}(x)|\lesssim \|\sigma\|_{N}\ang{x}^{-N},
\end{align*}
with implied constants only depending on $N$. 
\item Assume that $d\geq 2$ and define, for $t\in\mathbb{R}$, the reduced symbol 
\begin{align*}
\mathbb{R}^{d-1}\ni\xi\mapsto\sigma_t(\xi):=\sigma(t,\xi).
\end{align*}
Then we have that $\sigma_t\in \mathsf{W}^{\infty,1}(\mathbb{R}^{d-1})$, for all $t\in\mathbb{R}$. Moreover, for every $N\in\mathbb{N}_0$, it holds that $(t\mapsto \sigma_t)\in \mathsf{L}^1\big(\mathbb{R}, \mathsf{W}^{N,1}(\mathbb{R}^{d-1})\big)\cap \mathsf{C}\big(\mathbb{R}, \mathsf{W}^{N,1}(\mathbb{R}^{d-1})\big)$.
\end{enumerate}
\end{lemma}
\begin{proof}
Using the fact that $\sigma\in\mathsf{W}^{\infty,1}(\mathbb{R}^d)$ and integrating by parts we get that
\begin{align*}
|\check{\sigma}(x)|&\lesssim \big|\int d\xi\, e^{ix\cdot \xi}\sigma(\xi)\big|\nonumber\\
&=\big|\int d\xi\, \sigma(\xi)\Big[\frac{1-ix\cdot \nabla_{\xi}}{1+x^2}\Big]^N e^{ix\cdot\xi}\big| \nonumber\\
&=\big|\int d\xi\, e^{ix\cdot \xi}\Big[\frac{1+ix\cdot \nabla_{\xi}}{1+x^2}\Big]^N \sigma(\xi)\big|\nonumber\\
&\lesssim \|\sigma\|_N \ang{x}^{-N},
\end{align*}
where the implied constants only depend on $N$. For the proof of the second part of the statement notice that, since $\sigma\in \mathsf{L}^1(\mathbb{R}^d)$, there is some $t_0\in\mathbb{R}$ such that $\sigma_{t_0}\in \mathsf{L}^1(\mathbb{R}^{d-1})$. This in turn implies that
\begin{align}\label{eq:sigmatL1Linfinity}
\sup\limits_{t\in\mathbb{R}}\|\sigma_t\|_{\mathsf{L}^1(\mathbb{R}^{d-1})}&\leq \|\sup\limits_{t\in\mathbb{R}}|\sigma_t|\|_{\mathsf{L}^1(\mathbb{R}^{d-1})}=\int\limits_{\mathbb{R}^{d-1}}d\xi\, \sup\limits_{t\in\mathbb{R}}\Big|\sigma_{t_0}(\xi)+\int\limits_{t_0}^tds\, \partial_s \sigma_s(\xi)\Big|
\nonumber\\
&\leq \|\sigma_{t_0}\|_{\mathsf{L}^1(\mathbb{R}^{d-1})}+\|\partial_t\sigma\|_{\mathsf{L}^1(\mathbb{R}^{d})}<\infty,
\end{align}
i.e.~$\sigma_t\in \mathsf{L}^1(\mathbb{R}^{d-1})$ for all $t$. Moreover, the fact that $\sigma\in\mathsf{L}^1(\mathbb{R}^d)\cap \mathsf{C}(\mathbb{R}^d)$ and the uniform bound \eqref{eq:sigmatL1Linfinity} ensure that $(t\mapsto \sigma_t)\in \mathsf{L}^1\big(\mathbb{R},\mathsf{L}^1(\mathbb{R}^{d-1})\big)\cap \mathsf{C}\big(\mathbb{R},\mathsf{L}^1(\mathbb{R}^{d-1})\big)$. The analogous statements for derivatives of $\sigma_t$ follow along the same lines. This finishes the proof of the lemma.
\end{proof}
For $\sigma\in \mathsf{W}^{\infty,1}(\mathbb{R}^d)$, the off-diagonal decay of the kernel $A(x,y)=\check{\sigma}(x-y)$, see Lemma \ref{lem:sigma}, and the continuity of $\check{\sigma}$ imply the following lemma. Its rather technical proof is omitted.
\begin{lemma}\label{lem:kernelestimates}
Let $\sigma\in\mathsf{W}^{\infty,1}(\mathbb{R}^d)$ and let $h:\mathbb{C}\to\mathbb{C}$ be an entire function with $h(0)=0$. Then for any open set $G\subseteq\mathbb{R}^d$ the operator kernel 
\begin{align*}
(x,y)\mapsto h(A_G)(x,y)
\end{align*}
is a continuous function on $G\times G$.
\end{lemma}

\subsection{Localisation estimates}
Throughout this subsection, let $\sigma\in\mathsf{W}^{\infty,1}(\mathbb{R}^d)$ and let $h:\mathbb{C}\to\mathbb{C}$ be an entire function that vanishes to second order at $z=0$. One of the main tools for proving Theorem~\ref{thm:abstractasympt} is the next proposition. It is of similar spirit as \cite[Thm. 2.5]{Dietlein2018}, which was recently established in the context of ergodic Schrödinger operators.
\begin{proposition}\label{prop:trnormdecay}
Suppose that $\Lambda\subseteq \Omega\subseteq \mathbb{R}^d$ and let $a,b\in \mathbb{R}^d$. Then, for any $N\in\mathbb{N}$, there exists a constant $C_{h,\sigma,N}\geq 0$ such that 
\begin{align}\label{eq:trnormorderedsubsets}
\|\chi_{Q_a\cap \Lambda}\big[h(A_\Lambda)-h(A_\Omega)\big]\chi_{Q_b}\|_1 \leq C_{h,\sigma,N} \ang{\dist(a,\Omega\setminus \Lambda)}^{-N}\ang{\dist(b,\Omega\setminus \Lambda)}^{-N}\ang{a-b}^{-N}.
\end{align}
More precisely, if $h$ is given by the power series $h(z)=\sum\limits_{k=2}^\infty a_kz^k$, then the constant $C_{h,\sigma,N}$ may be bounded as
\begin{align}\label{eq:chsigmaNbound}
C_{h,\sigma,N}\lesssim \sum\limits_{k=2}^\infty k|a_k|\big[C_{N}\|\sigma\|_{2N+2d+2}\big]^k,
\end{align}
for some constant $C_N\geq 0$ and implied constant only depending on $N$.
\end{proposition}
\begin{remark}
\begin{enumerate}
\item Unlike in \cite[Thm. 2.5]{Dietlein2018}, we do not require convexity of the set $\Omega$.
\item In \cite{Dietlein2018} the author deduced their result with the help of an a-priori Schatten quasi-norm bound in $\mathfrak{S}_q$ for some $q<1$, see \cite[Eq. (2.3)]{Dietlein2018}. In the special case of Wiener-Hopf-operators, this a-priori bound reduces to
\begin{align}\label{eq:apriorischattenqbound}
\sup\limits_{a,b\in\mathbb{R}^d}\|\chi_{Q_a} A\chi_{Q_b}\|_q < \infty.
\end{align}
This estimate holds if, in addition to $\sigma\in\mathsf{W}^{\infty,1}(\mathbb{R}^2)$, we suppose that $\sigma\in \mathsf{L}^p(\mathbb{R}^2)$ for some $p\in (0,q)$, see \cite[Ch. 11, Thm. 13]{BS}. However, we prefer not to assume any additional decay on $\sigma$. Instead, we exploit the basic Hilbert Schmidt bound \eqref{eq:HSestonunitcubes} on unit cubes from Lemma \ref{lem:trnormestoffdiag} below.
\item  The mild decay assumptions on the symbol $\sigma$ are compensated by assuming that the test function $h$ vanishes to second order at $z=0$. This assumption is sufficient to prove Theorem \ref{thm:abstractasympt}: we will exclusively apply Proposition \ref{prop:trnormdecay} to the function $h_1$, see \eqref{eq:defh1}.
\end{enumerate}
\end{remark}
Proposition \ref{prop:trnormdecay} follows from approximation of the test function $h$ by polynomials and the next lemma. 
\begin{lemma}\label{lem:trnormestoffdiag}
Let $a,b\in\mathbb{R}^d$. Then, for all $N\in\mathbb{N}$, there exists constants $c_N\geq \tilde{c}_N\geq 0$ such that
\begin{align}\label{eq:HSestonunitcubes}
\|\chi_{Q_a} A\chi_{Q_b}\|\leq \|\chi_{Q_a} A\chi_{Q_b}\|_2\leq \tilde{c}_N\|\sigma\|_N\ang{a-b}^{-N},
\end{align}
and such that, for all $k\in\mathbb{N}\setminus \lbrace 1\rbrace$, $p\in\lbrace 1,2\rbrace$, 
\begin{align}\label{eq:polynomialHSest}
\sup\limits_{G\subseteq\mathbb{R}^d}\|\chi_{Q_a}[A_G]^k\chi_{Q_b}\|_p\leq \big[c_N\|\sigma\|_{N+d+1}\big]^k\ang{a-b}^{-N}.
\end{align}
\end{lemma}
\begin{proof}
The estimate \eqref{eq:HSestonunitcubes} is a direct consequence of Lemma \ref{lem:sigma}. To prove \eqref{eq:polynomialHSest} define, for all $N\geq 1$, the constants
\begin{align*}
\tilde{c}_{N,\sigma}:=\tilde{c}_N\sum\limits_{|\alpha|\leq N}\|\partial^\alpha \sigma\|_{\mathsf{L}^1(\mathbb{R}^d)},
\end{align*}
and set
\begin{align}\label{eq:supsum}
c_d':=\sup\limits_{|x|\leq 1}\sum\limits_{y\in\mathbb{Z}^d}\ang{y+x}^{-d-1}<\infty.
\end{align}
Let $N\in\mathbb{N}$, $k\geq 2$, $G\subseteq\mathbb{R}^d$, and $M:=N+d+1$. Then \eqref{eq:HSestonunitcubes} implies that, for $p\in\lbrace 1,2\rbrace$ and for all $a,b\in\mathbb{R}^d$,
\begin{align}\label{eq:HSestrestrpowers}
\|\chi_{Q_a} [A_G]^k \chi_{Q_b}\|_p&\leq \!\!\sum\limits_{y_1,\dots,y_{k-1}\in\mathbb{Z}^d}\!\!\!\!\!\! \|\chi_{Q_a} A\chi_{Q_{y_1}}\|_2 \|\chi_{Q_{y_1}} A\chi_{Q_{y_2}}\| \cdot \dots \cdot \|\chi_{Q_{k-2}} A\chi_{Q_{y_{k-1}}}\|\|\chi_{Q_{y_{k-1}}} A\chi_{Q_{b}}\|_2\nonumber\\
&\leq \tilde{c}_{M,\sigma}^k\!\!\!\!\sum\limits_{y_1,\dots,y_{k-1}\in\mathbb{Z}^d}\!\! \ang{a-y_1}^{-M}\ang{y_1-y_2}^{-M}\cdot \dots \cdot \ang{y_{k-2}-y_{k-1}}^{-M}\ang{y_{k-1}-b}^{-M}\nonumber\\
&\leq \tilde{c}_{M,\sigma}^k \big[2^{N/2}c_d'\big]^{k-1} \ang{a-b}^{-N}.
\end{align}
Here we have used Peetre's inequality,
\begin{align*}
\ang{x-y}^N\ang{y-z}^N\geq 2^{-N/2}\ang{x-z}^N,\ \text{for}\ x,y,z\in\mathbb{R}^d,
\end{align*} 
and the definition of $c_d'$, see \eqref{eq:supsum}. Setting
\begin{align*}
c_N:=\tilde{c}_{M}2^{N/2}c_d',
\end{align*}
\eqref{eq:polynomialHSest} follows and the proof of the lemma is complete.
\end{proof}

\begin{proof}[Proof of Proposition \ref{prop:trnormdecay}]
Let $h$ be an entire function of the form $h(z)=\sum\limits_{k=2}^\infty a_kz^k$. Then Lemma~\ref{lem:trnormestoffdiag} implies that
\begin{align}\label{eq:trnormdiffoffdiagdecay}
\|\chi_{Q_a\cap \Lambda}\big[h(A_\Lambda)-h(A_\Omega)\big]\chi_{Q_b}\|_1 \leq C_{h,\sigma,N}'\ang{a-b}^{-N},
\end{align}
where 
\begin{align*}
C_{h,\sigma,N}':=2\sum\limits_{k\geq 2}|a_k|\big[c_N\|\sigma\|_{N+d+1}\big]^k,
\end{align*}
and $c_N$ is the constant in Lemma \ref{lem:trnormestoffdiag}. As we may interpolate with \eqref{eq:trnormdiffoffdiagdecay}, it suffices to show that
\begin{align}\label{eq:finalesttrnormh}
\|\chi_{Q_a\cap \Lambda}\big[h(A_\Lambda)-h(A_\Omega)\big]\chi_{Q_b}\|_1 \leq C_{h,\sigma,N}''\ang{\dist(a,\Omega\setminus\Lambda)}^{-N}\ang{\dist(b,\Omega\setminus\Lambda)}^{-N},
\end{align}
for an appropriate constant $C_{h,\sigma,N}''$. Again, we first prove \eqref{eq:finalesttrnormh} for monomials $h(z)=z^k$, $k\geq 2$. Defining for $m,n\in\mathbb{N}_0$ the operators
\begin{align*}
\tau_{mn}:=\chi_{Q_a}[A_\Lambda]^m\chi_\Lambda A\chi_{\Omega\setminus\Lambda}[A_\Omega]^n\chi_{Q_b},
\end{align*}
one gets that
\begin{align}\label{eq:polydiff}
\chi_{Q_a\cap\Lambda}\big([A_\Lambda]^k-[A_\Omega]^k\big)\chi_{Q_b}&=\sum\limits_{l=0}^{k-1}\chi_{Q_a\cap \Lambda}[A_\Lambda]^{k-l-1}(A_\Lambda-A_\Omega)[A_\Omega]^{l}\chi_{Q_b}\nonumber\\
&=-\chi_\Lambda\sum\limits_{l=0}^{k-1}\tau_{k-l-1,l}.
\end{align}
Fix the numbers $M$ and $M'$, depending on $N$ and $d$: 
\begin{align*}
M:=N+d+1,\ M':=N+2d+2=M+d+1.
\end{align*}
Moreover, define for any set $G\subseteq \mathbb{R}^d$ the corresponding lattice point neighbourhood
\begin{align}\label{eq:deflatticepointneighb}
G_+:=\lbrace y\in \mathbb{Z}^d: Q_y\cap G\neq \emptyset\rbrace.
\end{align}
We apply Lemma \ref{lem:trnormestoffdiag} and estimate as in \eqref{eq:HSestrestrpowers} to deduce that, for $m,n\in\mathbb{N}$,
\begin{align}\label{eq:taumnest}
\|\tau_{mn}\|_1&\leq \sum\limits_{\substack{x\in \Lambda_+\\y\in (\Omega\setminus\Lambda)_+}}\|\chi_{Q_a} [A_\Lambda]^m\chi_{Q_x}\|_2\|\chi_{Q_x} A\chi_{Q_y}\| \|\chi_{Q_y} [A_\Omega]^n\chi_{Q_b}\|_2\nonumber\\
&\leq \sum\limits_{\substack{x\in \Lambda_+ \\ y\in (\Omega\setminus\Lambda)_+}}\big[c_{M}\|\sigma\|_{M'}\big]^{n+m+1}\ang{a-x}^{-M}\ang{x-y}^{-M}\ang{y-b}^{-M}\nonumber\\
&\lesssim \big[c_{M}\|\sigma\|_{M'}\big]^{n+m+1}\ang{\dist(a,\Omega\setminus\Lambda)}^{-N}\ang{\dist(b,\Omega\setminus\Lambda)}^{-N}.
\end{align}
Here, the implied constants only depend on $N$.
Similarly, we estimate for $n\geq 1$,
\begin{align}\label{eq:taum0est}
\|\tau_{0n}\|_1&\leq \sum\limits_{y\in (\Omega\setminus\Lambda)_+}\|\chi_{Q_a}A\chi_{Q_y}\|_2\|\chi_{Q_y}[A_\Omega]^n\chi_{Q_b}\|_2\nonumber\\
&\leq \sum\limits_{y\in (\Omega\setminus\Lambda)_+} \big[c_{M}\|\sigma\|_{M'}\big]^{n+1}\ang{a-y}^{-M}\ang{y-b}^{-M}\nonumber\\
&\lesssim \big[c_{M}\|\sigma\|_{M'}\big]^{n+1}\ang{\dist(a,\Omega\setminus\Lambda)}^{-N}\ang{\dist(b,\Omega\setminus\Lambda)}^{-N}.
\end{align}
In case $Q_b\cap \Omega\setminus\Lambda=\emptyset$ one has that $\tau_{m0}=0$, hence combining \eqref{eq:polydiff}, \eqref{eq:taumnest}, and \eqref{eq:taum0est} gives
\begin{align}\label{eq:finalpolyest}
\|\chi_{Q_a\cap \Lambda}\big([A_\Omega]^k-[A_\Lambda]^k\big)\chi_{Q_b}\|_1\lesssim k \big[c_{M}\|\sigma\|_{M'}\big]^k \ang{\dist(a,\Omega\setminus\Lambda)}^{-N}\ang{\dist(b,\Omega\setminus\Lambda)}^{-N},
\end{align} 
with implied constants only depending on $N$. If $Q_b\cap \Omega\setminus\Lambda\neq \emptyset$, we estimate 
\begin{align*}
\|\tau_{m0}\|_1&\leq \|\chi_{Q_a}[A_\Lambda]^m\chi_\Lambda A\chi_{Q_b}\|_1\nonumber\\
&\leq \sum\limits_{x\in \Lambda_+}\|\chi_{Q_a} [A_\Lambda]^m\chi_{Q_x}\|_2\|\chi_{Q_x}A\chi_{Q_b}\|_2 \nonumber\\
&\leq \sum\limits_{x\in \Lambda_+}\big[c_{M}\|\sigma\|_{M'}\big]^{m+1}\ang{a-x}^{-M}\ang{x-b}^{-M}\nonumber\\
&\lesssim \big[c_{M}\|\sigma\|_{M'}\big]^{m+1}\ang{\dist(a,\Omega\setminus\Lambda)}^{-N}, 
\end{align*} 
which together with \eqref{eq:polydiff}, \eqref{eq:taumnest}, and \eqref{eq:taum0est} again implies \eqref{eq:finalpolyest}. The extension of Estimate \eqref{eq:finalpolyest} to entire functions $h$ of the form $h(z)=\sum\limits_{k=2}^\infty a_kz^k$, and an interpolation with \eqref{eq:trnormdiffoffdiagdecay} finishes the proof of the proposition.
\end{proof}
Proposition \ref{prop:trnormdecay} implies two corollaries, which will be useful in applications. For example, it follows from Corollary \ref{cor:traceclassifpolybdd} that the coefficient $a_1(\nu_E)$, see \eqref{eq:defa1nuE}, is well-defined.
\begin{corollary}\label{cor:traceclassifpolybdd}
Suppose that the sets $\mathsf{M}$, $\Lambda$, $\Omega\subseteq \mathbb{R}^d$ satisfy 
\begin{align*}
\mathsf{M}\subseteq \Lambda\cap\Omega.
\end{align*}
Moreover, assume that there exists $\beta\geq 0$ and a constant $C_\beta\geq 0$ such that, for all $r>0$,
\begin{align}\label{eq:numberofpointspolynomiallybounded}
\sharp \lbrace x\in \mathsf{M}_+:\ \dist(x,\Lambda\triangle\Omega)\leq r\rbrace \leq C_\beta\ang{r}^\beta,
\end{align}
where the set $\mathsf{M}_+\subseteq\mathbb{Z}^d$ is defined in \eqref{eq:deflatticepointneighb}.
\\
Then we have that
\begin{align*}
\chi_{\mathsf{M}}\big[h(A_\Lambda)-h(A_\Omega)\big]\in \mathfrak{S}_1.
\end{align*}
\end{corollary}
\begin{proof}
An application of the triangle inequality shows that we may restrict ourselves to the case that $\Lambda\subseteq \Omega$. Applying Proposition \ref{prop:trnormdecay} for $N=d+\beta+1$ and the assumption $\mathsf{M}\subseteq \Lambda$, one gets that
\begin{align}\label{eq:estproofofcor}
\big\|\chi_{\mathsf{M}}\big[h(A_\Lambda)-h(A_\Omega)\big]\big\|_1&\leq \sum\limits_{\substack{a\in \mathsf{M}_+\\ b\in\mathbb{Z}^d}}\big\|\chi_{Q_a\cap \Lambda}\big[h(A_\Lambda)-h(A_\Omega)\big]\chi_{Q_b}\big\|_1 \nonumber\\
&\lesssim \sum\limits_{\substack{a\in \mathsf{M}_+\\ b\in\mathbb{Z}^d}} \ang{\dist(a,\Omega\setminus\Lambda)}^{-d-\beta-1}\ang{a-b}^{-d-\beta-1} \nonumber\\
&\lesssim \sum\limits_{a\in \mathsf{M}_+}\ang{\dist(a,\Omega\setminus\Lambda)}^{-d-\beta-1} \nonumber\\
&\leq \sum\limits_{k=0}^\infty\sum\limits_{\substack{a\in \mathsf{M}_+ \\ k\leq \dist(a,\Omega\setminus\Lambda)\leq k+1}}\hspace{-0.8cm} \ang{k}^{-d-\beta-1} \nonumber\\
&\lesssim \sum\limits_{k=0}^\infty C_\beta\ang{k+1}^\beta \ang{k}^{-d-\beta-1} <\infty,
\end{align}
where the implied constants depend on $\beta$, $h$,  and $\sigma$. This finishes the proof of the corollary.
\end{proof}
The next corollary treats $L$-dependent sets $\mathsf{M}$, $\Lambda$, and $\Omega$. Here, the dependence on $L$ does not need to be linear, unlike for scaled sets. The corollary gives sufficient conditions under which the spatial restriction of the operator $h(A_{\Lambda})$ to $\mathsf{M}$ may be replaced by the corresponding restriction of $h(A_{\Omega})$, with a super-polynomially small error in trace norm, as $L\to\infty$. 
\begin{corollary}\label{cor:Ldependenttrnormest}
Let $\mathsf{M},\Lambda,\Omega\subseteq \mathbb{R}^d$ be sets that all possibly depend on the parameter $L\geq 1$. Suppose that
\begin{align}\label{eq:distMsymmdifforderL}
\mathsf{M}\subseteq \Lambda\cap \Omega \quad \text{and} \quad \dist(\mathsf{M},\Lambda\triangle \Omega )\gtrsim L.
\end{align}
Moreover, assume that there exists some $\beta\geq 0$ and a constant $C_\beta\geq 0$, independent of $L$, such that at least one of the following conditions is satisfied:
\begin{enumerate}[(i)]
\item \label{numberE+bounded} $\sharp \mathsf{M}_+ \leq C_\beta L^\beta$.
\item \label{previousestholds} Estimate \eqref{eq:numberofpointspolynomiallybounded} holds.
\end{enumerate}
Then one has that 
\begin{align*}
\big\| \chi_\mathsf{M} \big[h(A_\Lambda)-h(A_\Omega)\big]\big\|_1=\mathcal{O}(L^{-\infty}),
\end{align*}
as $L\to\infty$.
\end{corollary}
\begin{proof}
As in the proof of Corollary \ref{cor:traceclassifpolybdd}, we may assume that $\Lambda\subseteq \Omega $. Moreover, similarly as in \eqref{eq:estproofofcor}, an application of Proposition \ref{prop:trnormdecay} yields
\begin{align*}
\big\| \chi_\mathsf{M} \big[h(A_\Lambda)-h(A_\Omega)\big]\big\|_1&\lesssim \sum\limits_{a\in \mathsf{M}_+} \ang{\dist(a,\Omega\setminus\Lambda)}^{-N},
\end{align*}
with implied constant depending on  $h$, $\sigma$, and $N\geq d+1$. If the estimate \eqref{numberE+bounded} holds, then one easily concludes with \eqref{eq:distMsymmdifforderL} that
\begin{align*}
\big\| \chi_\mathsf{M} \big[h(A_\Lambda)-h(A_\Omega)\big]\big\|_1\lesssim L^{\beta-N},
\end{align*}
with implied constant depending on $\beta, N, h$, and $\sigma$. Assuming \eqref{previousestholds} instead, we obtain as in the proof of Corollary \ref{cor:traceclassifpolybdd} that
\begin{align*}
\big\| \chi_\mathsf{M} \big[h(A_\Lambda)-h(A_\Omega)\big]\big\|_1&\lesssim L^{-N/2} \sum\limits_{a\in \mathsf{M}_+} \ang{\dist(a,\Omega\setminus\Lambda)}^{-N/2} \nonumber\\
&\lesssim L^{-N/2},
\end{align*}
where we chose $N\geq 2(d+\beta+1)$ and the implied constants depend on $N$, $\beta$, $h$, $\sigma$, and the constant in \eqref{eq:numberofpointspolynomiallybounded}. This finishes the proof of the corollary.
\end{proof}
\section{Proof of Theorem \ref{thm:abstractasympt}: localisation to the corners of $P$}
\label{sec:redasymptcorners}

Fix $\epsilon>0$ to be chosen later and recall the definition \eqref{eq:defV} of $\mathcal{V}=\mathcal{V}^{(\epsilon)}$, the one-sided $\epsilon$--neighbourhood of $\partial P$. As indicated in Subsection \ref{subsec:strategyproofs}, we split $\mathcal{V}$ into (almost) disjoint sets $\mathcal{N}(X)$, $X\in\Xi(P)$, such that $\mathcal{N}(X)$ contains the part of $\mathcal{V}$ close to the vertex $X$, see Figure~\ref{fig:1} on page~\pageref{fig:1}. This induces a corresponding partition of $\mathcal{V}_L=L\cdot\mathcal{V}$.

\subsection{Partition of $\mathcal{V}_L$} 
Fix a vertex $X\in\Xi(P)$ and recall that its adjacent edges are named $E^{(1)}(X)$ and $E^{(2)}(X)$, see Subsection \ref{subsec:notationPandcoeff}. It will be convenient to introduce the following two choices for the unit normal and the unit tangent vector at $X$:
\begin{align}
\begin{aligned}\label{eq:deftauXnuX}
(\tau_X^{(1)},\nu_X^{(1)})&:=(-\tau_{E^{(1)}(X)},\nu_{E^{(1)}(X)}),\\
(\tau_X^{(2)},\nu_X^{(2)})&:=(\tau_{E^{(2)}(X)},\nu_{E^{(2)}(X)}).
\end{aligned}
\end{align}
This definition ensures that $\tau^{(1)}_X$ and $\tau^{(2)}_X$, considered as vectors at $X$, point into the direction of the edges $E^{(1)}(X)$ and $E^{(2)}(X)$, respectively. For $j=1,2$, define the tubes
\begin{align}\label{eq:defTj}
T^{(j)}(X):&=\lbrace a\tau_X^{(j)}+b\nu_X^{(j)}: \ (a,b)\in [0,\tfrac{|E^{(j)}(X)|}{2}]\times [0,\epsilon]\rbrace,
\end{align}
and set
\begin{align}\label{eq:defN}
\mathsf{N}(X):=\big[T^{(1)}(X)\cup T^{(2)}(X)\cup B_{\epsilon}(0)\big]\cap C(X),
\end{align}
see Figure~\ref{fig:2} below. Then $\mathsf{N}(X)$ is a corner-neighbourhood of $0\in \Xi(P-X)$ and we define the corresponding neighbourhood at $X\in\Xi(P)$ by
\begin{align}\label{eq:defcalNX}
\mathcal{N}(X):=X+\mathsf{N}(X).
\end{align}
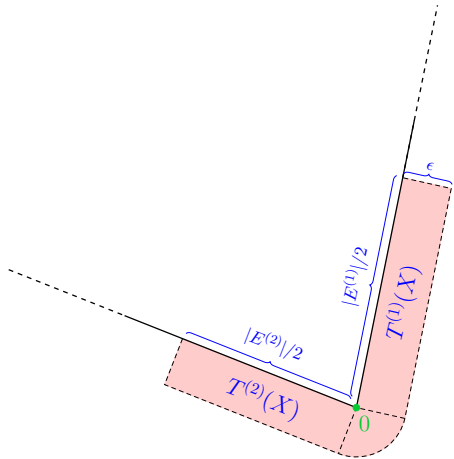
\begin{figure}[!b]
\hspace{-4cm}
\resizebox{6cm}{!}{
\begin{tikzpicture}
	    \filldraw[red!20!white]
(0,0) -- (-3.75,1.5) -- (-4.137,0.5) -- ([shift=(-111:1.072)] 0, 0) arc(-111:-13:1.072cm) -- (1.046,-0.235)--(2.046,4.765)--(1,5)--cycle;
		\draw[densely dashed] (1.046,-0.235)--(0,0)--(-0.387,-1);
		\draw[densely dashed]
		(-3.75,1.5)--(-4.137,0.5) -- ([shift=(-111:1.072)] 0, 0) arc(-111:-13:1.072cm) -- (2.046,4.765)--(1,5);
	\draw[thick] (-5,2)--(0,0)--(1.25,6.25);
	\draw[thick,dashed] 
		(-7.5,3) -- (-5,2) (1,5)--(1.75,8.75);
	\filldraw[green!80!blue] (0,0) circle(2pt);
	\node[green!80!blue] at (0.18,-0.35){\Large $0$};
	\node[blue,very thick,rotate=-21] at (-2,0.2){\Large $T^{(2)}(X)$};
	\node[blue,very thick,rotate=77] at (1,2.3){\Large $T^{(1)}(X)$};
	\draw [blue,decorate,decoration={brace,amplitude=3pt},yshift=4pt](-3.675,1.47)--(-0.1225,0.049) node [blue,midway,above=4pt,xshift=0pt,rotate=-21] {$|E^{(2)}|/2$};
	\draw [blue,decorate,decoration={brace,amplitude=3pt},xshift=-2.5pt, yshift=7pt](0,0)--(0.96,4.8) node [blue,midway,above=4pt,xshift=-2pt,yshift=-0.4pt,rotate=77] {$|E^{(1)}|/2$};
		\draw [blue,decorate,decoration={brace,amplitude=3pt},xshift=0.5pt, yshift=2.5pt](1,5)--(2.046,4.765) node
	[blue,midway,above=2pt,xshift=0pt,rotate=-13] {$\epsilon$};
\end{tikzpicture}
}
\caption{The neighbourhood $\mathsf{N}(X)$ for a vertex $X\in\Xi_>(P)$}\label{fig:2}
\end{figure}
Combining the scaled neighbourhoods $\mathcal{N}_L(X)=L\cdot\mathcal{N}(X)$, we arrive at the partition
\begin{align}\label{eq:NLunionofcornerneighb}
\mathcal{V}_L=\bigcup\limits_{X\in\Xi(P)}\mathcal{N}_L(X).
\end{align}
At this point we choose $\epsilon>0$ small enough such that the union \eqref{eq:NLunionofcornerneighb} is disjoint up to sets of zero two-dimensional Lebesgue measure.

\subsection{Reduction to individual corner contributions} 
\label{subsec:proofofabstrthm1}
As in the formulation of Theorem \ref{thm:abstractasympt}, let $\sigma\in\mathsf{W}^{\infty,1}(\mathbb{R}^2)$ and assume that $h$ is an entire function with $h(0)=0$. Notice that due to Lemma~\ref{lem:troponboundeddomains} the operators $h(A_{P_L})$ and $\chi_{P_L}h(A)\chi_{P_L}$ are trace class, the trace of the latter operator being computed in \eqref{eq:bulkapproxtrace}. This gives us the leading order term in the asymptotics \eqref{eq:abstractthmasymptformula}:
\begin{align}\label{eq:leadingorderasympt}
\tr h(A_{P_L})= L^2c_2+ \tr\big(\chi_{P_L}\big[h(A_{P_L})-h(A)\big]\chi_{P_L}\big).
\end{align}
Moreover, it follows from Corollary \ref{cor:traceclassifpolybdd} that
\begin{align*}
\chi_{P_L}\big[h_1(A_{P_L})-h_1(A)\big]\in \mathfrak{S}_1,
\end{align*}
where we recall the definition of the function $h_1(z)=h(z)-zh'(0)$. By construction, we have that 
\begin{align*}
\dist(P_L\setminus \mathcal{V}_L\, ,\,\mathbb{R}^2\setminus P_L)\gtrsim L,
\end{align*}
hence Corollary \ref{cor:Ldependenttrnormest} with Assumption \eqref{numberE+bounded} implies that
\begin{align*}
\tr\big(\chi_{P_L}\big[h(A_{P_L})-h(A)\big]\chi_{P_L}\big)&=\tr\big(\chi_{P_L}\big[h_1(A_{P_L})-h_1(A)\big]\big)\nonumber\\
&=\tr\big(\chi_{\mathcal{V}_L}\big[h_1(A_{P_L})-h_1(A)\big]\big)+\mathcal{O}(L^{-\infty}).
\end{align*}
In particular, we may from now on assume that $h$ vanishes to second order at $z=0$, such that $h_1=h$. Also, we have reduced the proof of Theorem \ref{thm:abstractasympt} to computing the asymptotics of
\begin{align*}
\tr\big(\chi_{\mathcal{V}_L}\big[h(A_{P_L})-h(A)\big]\big)=\sum\limits_{X\in\Xi(P)}\tr\big(\chi_{\mathcal{N}_L(X)}\big[h(A_{P_L})-h(A)\big]\big),
\end{align*}
employing \eqref{eq:NLunionofcornerneighb} for the latter equality. For fixed $X\in\Xi(P)$, the translation-invariance of the operator $A$ implies that
\begin{align*}
\tr\big(\chi_{\mathcal{N}_L(X)}\big[h(A_{P_L})-h(A)\big]\big)=\tr\big(\chi_{\mathsf{N}_L(X)}\big[h(A_{(P-X)_L})-h(A)\big]\big),
\end{align*}
with $\mathsf{N}(X)=\mathcal{N}(X)-X$, see \eqref{eq:defcalNX}. Moreover, it is not difficult to see that 
\begin{align*}
\dist\big(\mathsf{N}_L(X)\,,\,C(X)\triangle (P-X)_L\big)\gtrsim L.
\end{align*}
Hence, Corollary \ref{cor:Ldependenttrnormest} with Assumption \eqref{numberE+bounded} yields that
\begin{align*}
\tr\big(\chi_{\mathsf{N}_L(X)}\big[h(A_{(P-X)_L})-h(A)\big]\big)=\tr\big(\chi_{\mathsf{N}_L(X)}\big[h(A_{C(X)})-h(A)\big]\big)+\mathcal{O}(L^{-\infty}).
\end{align*}
We emphasise that the trace
\begin{align}\label{eq:localcornertrace}
\tr\big(\chi_{\mathsf{N}_L(X)}\big[h(A_{C(X)})-h(A)\big]\big)
\end{align}
depends on the polygon $P$ only via the directions $\tau_X^{(j)}$, $j=1,2$, and the length of the edges adjacent to $X$. To compute its asymptotics for each $X\in\Xi(P)$ is the object of the next section.

\section{Proof of Theorem \ref{thm:abstractasympt}: asymptotics for a fixed corner of $P$}\label{sec:sectopslocaltrasympt}

Throughout this section we fix a vertex $X\in\Xi(P)$. In particular, we shall omit all arguments, sub- and superscripts ``$(X)$''; for instance, we will write $C=C(X)$ and $\mathsf{N}_L=\mathsf{N}_L(X)$. As before, let $\sigma\in\mathsf{W}^{\infty,1}(\mathbb{R}^2)$ and  assume that $h=h_1$ is an entire function that vanishes to second order at $z=0$. The main purpose of this section is to obtain an asymptotic formula for \eqref{eq:localcornertrace}, which will complete the proof of Theorem \ref{thm:abstractasympt}.

\subsection{The $L$-term in the asymptotics} In the smooth boundary case, the sub-leading order term in the asymptotics \eqref{eq:trasymptsmoothboundary} is (at least morally) obtained via approximation of the operator $h(A_{\Omega_L})$ by half-space operators: around $x\in \partial\Omega_L$, the operator $h(A_{\Omega_L})$ is replaced by $h(A_{H_x})$ where $H_x$ is the half-space approximation of $\Omega_L$ at $x$. Similarly, the half-spaces $H^{(1)}$ and $H^{(2)}$, see \eqref{eq:defHjX}, locally model the sector $C$ in \eqref{eq:localcornertrace}, as long as one stays away from the apex of the sector. Thus, to get a first-order approximation to \eqref{eq:localcornertrace}, the strategy is to replace the sector $C$ by the half-space $H^{(j)}$, $j=1,2$, on the part of $\mathsf{N}_L$ close to $\partial H^{(j)}\cap\partial C$. This philosophy was used for right-angled cones in \cite{Thorsen1996} and \cite{Dietlein2018}. In the course of this section we will thus prove that
\begin{align}\label{eq:Lterm}
\tr\big(\chi_{\mathsf{N}_L}\big[h(A_C)-h(A)\big]\big)=\sum\limits_{j=1}^2 \tr\big(\chi_{T_L^{(j)}}\big[h(A_{H^{(j)}})-h(A)\big]\big) + \mathcal{O}(1),
\end{align}
as $L\to\infty$, see \eqref{eq:defTj} and Figure~\ref{fig:2} above for the definition of $T^{(j)}$. Here, the $\mathcal{O}(1)$-term contains the corner contribution at $X$ to the coefficient $c_0$ and a super-polynomial error in $L$. The approximation \eqref{eq:Lterm} is useful since the invariance of the operator $h(A_{H^{(j)}})-h(A)$ with respect to translations along the edge $E^{(j)}$ can be applied to scale out the length of the tube $T_L^{(j)}$. This is demonstrated in the next lemma, which hence provides the $L$-term in the asymptotics of \eqref{eq:localcornertrace}.
\begin{lemma}\label{lem:Lterm}
Let $j\in \lbrace 1,2\rbrace$ and set $S^{(j)}:=S_{E^{(j)}}$, compare with \eqref{eq:defSE}. Then one has that
\begin{align*}
\tr\big(\chi_{T_L^{(j)}}\big[h(A_{H^{(j)}})-h(A)\big]\big)=L\tfrac{|E^{(j)}|}{2}\tr\big(\chi_{S^{(j)}}\big[h(A_{H^{(j)}})-h(A)\big]\big) + \mathcal{O}(L^{-\infty}).
\end{align*}
\end{lemma}
\begin{proof}
Fix $j\in \lbrace 1,2\rbrace$ and omit the superscript  ``$(j)$'' for the duration of the proof. Moreover, we may assume after a suitable rotation that $H=\mathbb{R}\times [0,\infty)$ and $S=[0,1]\times[0,\infty)$. Then it follows from Corollary \ref{cor:Ldependenttrnormest} that
\begin{align*}
\tr\big(\chi_{T_L}\big[h(A_{H})-h(A)\big]\big)=\tr\big(\chi_{L\tfrac{|E|}{2}\cdot S}\big[h(A_H)-h(A)\big]\big)+\mathcal{O}(L^{-\infty}).
\end{align*}
Here, the trace on the right-hand side is well-defined due to Corollary \ref{cor:traceclassifpolybdd}. Also, the invariance of the operator $h(A_H)-h(A)$ with respect to translations in the $x_1$-direction implies that, for all $x=(x_1,x_2)\in\mathbb{R}^2$,
\begin{align*}
(h(A_H)-h(A))(x_1,x_2;x_1,x_2)=(h(A_H)-h(A))(0,x_2;0,x_2).
\end{align*}
Thus, a change of coordinates in the $x_1$-variable finishes the proof of the lemma.
\end{proof}

\subsection{Regularisation of sector operators} The key to finding the constant order term in the asymptotics of \eqref{eq:localcornertrace} is a trace-class regularisation of the sector operator $h(A_{C})$ with the help of the half-space operators $h(A_{H^{(j)}})$, $j=1,2$, and the full-space operator $h(A)$. This regularisation is given in the next proposition. For its proof we consider spatial restrictions of $h(A_{C})$ to different parts of the sector $C$ and then compare these to the operators $h(A_{H^{(j)}})$, $j=1,2$, or $h(A)$, depending on which part of the sector we localise to. In that respect we follow the ideas of \cite{Dietlein2018}. However, instead of only looking at a right-angled convex cone, we tackle sectors of any angle; in particular, we also deal with concave sectors. Moreover, our regularisation for convex sectors $C$, see \eqref{eq:regularisationconvex}, does not require a partition of $C$. At the same time, it is independent of the scaling parameter $L$, in contrast to the ones given in \cite[Thm. 2.2]{Dietlein2018}.
\begin{proposition}\label{prop:regularisationcornerop}
Let $L\geq 1$. If $X\in\Xi_<(P)$, then the operator 
\begin{align}\label{eq:regularisationconvex}
Z:=\chi_{C}\big[h(A_{C})-h(A_{H^{(1)}})-h(A_{H^{(2)}})+h(A)\big] 
\end{align}
is trace class with
\begin{align}\label{eq:trnormestconvcorner}
\|\chi_{\mathbb{R}^2\setminus B_L(0)}Z\|_1 = \mathcal{O}(L^{-\infty}),
\end{align}
as $L\to\infty$.
\\
If $X\in\Xi_>(P)$, then the operators 
\begin{align*}
Z_1&:=\chi_{H^{(1)}\cap H^{(2)}}\big[h(A_{C})-h(A)\big],\\
Z_2&:=\chi_{C\setminus H^{(1)}}\big[h(A_{C})-h(A_{H^{(2)}})\big],\\
Z_3&:=\chi_{C\setminus H^{(2)}}\big[h(A_{C})-h(A_{H^{(1)}})\big],
\end{align*}
are trace class and, for every $j=1,2,3$, one has that
\begin{align*}
\|\chi_{\mathbb{R}^2\setminus B_L(0)}Z_j\|_1=\mathcal{O}(L^{-\infty}),
\end{align*}
as $L\to\infty$.
\end{proposition}
\begin{proof}
As in the statement of the proposition we treat convex and concave corners separately. \\
\textit{Convex corners, i.e.~$X\in\Xi_<(P)$:}
we divide the semi-infinite sector $C$ into two halves, 
\begin{align*}
C_l&:=\lbrace y\in C: y\cdot (\nu^{(2)}-\nu^{(1)})\geq 0 \rbrace, \\
C_r&:=\lbrace y\in C: y\cdot (\nu^{(1)}-\nu^{(2)})\geq 0 \rbrace,
\end{align*}
where we recall the definition \eqref{eq:deftauXnuX} for $\nu^{(j)}=\nu_X^{(j)}$.
Then one can write
\begin{align*}
Z&=\chi_{C_l}\big[ h(A_C)-h(A_{H^{(1)}})\big]+\chi_{C_r}\big[ h(A_C)-h(A_{H^{(2)}})\big]\nonumber\\
&\ \ \ +\chi_{C_l}\big[ h(A)-h(A_{H^{(2)}})\big]+\chi_{C_r}\big[ h(A)-h(A_{H^{(1)}})\big].
\end{align*}
Thus, Corollary \ref{cor:traceclassifpolybdd} implies that the operator $Z$ is trace class since the estimate \eqref{eq:numberofpointspolynomiallybounded} with $\beta=1$ is easily checked for all involved sets. Moreover, applying the same splitting for $Z$, the bound \eqref{eq:trnormestconvcorner} follows from Corollary \ref{cor:Ldependenttrnormest}.\\
\textit{Concave corners, i.e.~$X\in \Xi_>(P)$:} in the concave case we may directly apply Corollaries \ref{cor:traceclassifpolybdd} and \ref{cor:Ldependenttrnormest} to the operators $Z_j$, $j=1,2,3$; no further partition is required. The claim follows as in the convex case, which finishes the proof of the proposition.
\end{proof}

\subsection{Contributions from non-right-angled corners}
In the next subsection we will apply the regularisation for the sector operator $h(A_{C})$ from Proposition \ref{prop:regularisationcornerop} to find the asymptotics of the trace \eqref{eq:localcornertrace}. As it turns out during this process, non-perpendicular edges $E^{(1)}$ and $E^{(2)}$ generate an extra term of constant order. Technically, this relies on the fact that the tubes $T^{(j)}$, which are responsible for the $L$-term in the asymptotics, see Lemma \ref{lem:Lterm}, are rectangles. In this sense, they are not compatible with interior angles $\gamma\notin \lbrace \tfrac{\pi}{2}, \tfrac{3\pi}{2}\rbrace$.
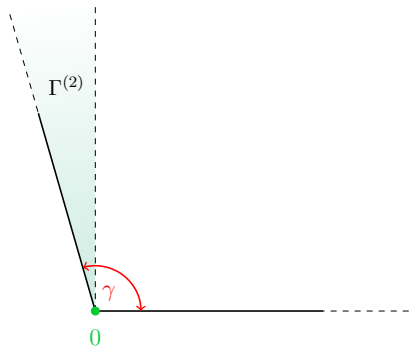
\begin{figure}[b!]
\resizebox{6cm}{!}{
\begin{tikzpicture}
\clip (-1.75,-1) rectangle (6,6);
      \shade[shading=radial, inner color=green!60!blue!20!white, outer color= white] (0,0) circle (5.5cm);
      \fill[color=white] (0,0)--(0,5.4) arc (90:-254:5.4)--(0,0);
    \draw[thick]
		(-1,3.5) coordinate (a) -- (0,0)
		coordinate (b) -- 
		(4,0) coordinate (c)
		pic["\large {\color{red}$\gamma$}", draw=red, <->, angle 
		eccentricity=0.5, angle radius=0.8cm]
		{angle=c--b--a};	
	\draw[dashed] (0,0) -- (0,5.4);
	\draw[dashed] 
		(-1.5,5.25) -- (-1,3.5); 
	\draw[dashed]		
		(5.5,0) -- (4,0);
	\filldraw[green!80!blue] (0,0) circle(2pt);
	\draw (0,-0.2) node[anchor=north, text=green!80!blue]{\large $0$};
	\node at (-0.5,4) {$\Gamma^{(2)}$};
\end{tikzpicture}
}
\vspace{-0.2cm}
\caption{The sector $\Gamma^{(2)}$ for $\gamma\in(\tfrac{\pi}{2},\pi)$}\label{fig:3}
\end{figure}

For the fixed vertex $X\in\Xi(P)$, introduce the following sectors, which depend on $j\in \lbrace 1,2\rbrace$, see Figure \ref{fig:3}:
\begin{align}\label{eq:defGammajX}
\Gamma^{(j)}:=\begin{cases}\lbrace a\tau^{(j)}+b\nu^{(j)}:\ 0\leq a < \cot(\gamma)b \rbrace,& \ \gamma\in(0,\tfrac{\pi}{2}]\cup (\pi,\tfrac{3\pi}{2}],\\
\lbrace a\tau^{(j)}+b\nu^{(j)}:\ \cot(\gamma)b< a \leq 0\rbrace,& \ \gamma\in[\tfrac{\pi}{2},\pi)\cup [\tfrac{3\pi}{2},2\pi).
\end{cases}
\end{align}
We will see in Subsection \ref{subsec:cornercontributions} that non-perpendicular edges $E^{(1)}$ and $E^{(2)}$ contribute the constants
\begin{align}\label{eq:traceGammaj}
\tr\big(\chi_{\Gamma^{(j)}}\big[h(A_{H^{(j)}})-h(A)\big]\big), \ j=1,2,
\end{align}
to the asymptotics of \eqref{eq:localcornertrace}. These traces are well-defined in view of Corollary \ref{cor:traceclassifpolybdd} and the following lemma provides an alternative characterisation of \eqref{eq:traceGammaj}.
\begin{lemma}\label{lem:notrectangularpart}
Let $X\in\Xi(P)$ be a vertex of $P$ and let 
$\Gamma^{(j)}$, $j=1,2$, be the sectors introduced in \eqref{eq:defGammajX}. Moreover, let $S^{(j)}$ be the strip of unit width defined in Lemma \ref{lem:Lterm}. Then we have that, for $j=1,2$,
\begin{align*}
\tr\big(\chi_{\Gamma^{(j)}(X)}\big[h(A_{H^{(j)}})-h(A)\big]\big)&=|\cot(\gamma)|\tr\big(\chi_{S^{(j)}} M(x\cdot \nu_{H^{(j)}})\big[h(A_{H^{(j)}})-h(A)\big]\big). 
\end{align*}
\end{lemma}
\begin{proof}
Fix $X\in\Xi(P)$. Without loss of generality suppose that $\gamma\in (0,\pi/2]$ and $j=2$, and, for the matter of readability, omit the superscript ``$(2)$''. The other cases can be reduced to this one via a symmetry argument. After a suitable rotation we may also assume that $H=\mathbb{R}\times [0,\infty)$, $\Gamma=\lbrace (x_1,x_2)\in H: 0\leq x_1\leq \cot(\gamma)x_2\rbrace$, and $S=[0,1]\times [0,\infty)$. Splitting the strip $S$ into unit cubes, one easily gets from Proposition \ref{prop:trnormdecay} that the operator
\begin{align*}
\chi_S M(x_2)\big[h(A_H)-h(A)\big]
\end{align*}
is trace-class. In view of Corollary \ref{cor:traceclassifpolybdd}, we likewise have that 
\begin{align}\label{eq:gammaoptraceclass}
\chi_\Gamma\big[h(A_H)-h(A)\big]\in\mathfrak{S}_1.
\end{align}
Furthermore, as in the proof of Lemma \ref{lem:Lterm} the invariance of the operator $h(A_H)-h(A)$ with respect to translations in the $x_1$-direction implies that, for all $x=(x_1,x_2)\in\mathbb{R}^2$,
\begin{align*}
(h(A_H)-h(A))(x_1,x_2;x_1,x_2)=(h(A_H)-h(A))(0,x_2;0,x_2).
\end{align*}
By Lemma \ref{lem:kernelestimates} this kernel is continuous on $\Gamma\times\Gamma\subset H\times H$, so \cite[Thm. 3.5]{Brislawn1988} and \eqref{eq:gammaoptraceclass} ensure that it is integrable on $\Gamma\times\Gamma$. Hence, we may apply Fubini's theorem to arrive at
\begin{align*}
\tr\big(\chi_\Gamma\big[h(A_H)-h(A)\big]\big)&=\int\limits_\Gamma dx_1dx_2\, (h(A_H)-h(A))(0,x_2;0,x_2)\\
&=\int\limits_0^\infty dx_2 \!\!\!\int\limits_0^{\cot(\gamma)x_2}\!\!\!\!\!\! dx_1\,(h(A_H)-h(A))(0,x_2;0,x_2) \nonumber\\
&=\cot(\gamma)\int\limits_0^\infty dx_2\, x_2\, (h(A_H)-h(A))(0,x_2;0,x_2)\nonumber\\
&=\cot(\gamma)\int\limits_0^1 dx_1\int\limits_0^\infty dx_2\, x_2\,(h(A_H)-h(A))(0,x_2;0,x_2)\nonumber\\
&=\cot(\gamma)\tr\big(\chi_S M(x_2) \big[h(A_H)-h(A)\big]\big).
\end{align*}
This finishes the proof of the lemma.
\end{proof}

\subsection{Complete asymptotics}
\label{subsec:cornercontributions}
Equipped with Proposition~\ref{prop:regularisationcornerop} and Lemmas~\ref{lem:Lterm} and \ref{lem:notrectangularpart}, we are now ready to extract the asymptotics from \eqref{eq:localcornertrace}. As the regularisation for the sector operators in Proposition~\ref{prop:regularisationcornerop} depends on the type of the sector, we naturally have to distinguish convex and concave corners of the polygon $P_L$. Propositions \ref{prop:convcorners} and \ref{prop:conccorners} contain the respective results.
\begin{proposition}[Convex corners]\label{prop:convcorners}
Let $X\in\Xi_<(P)$. Then we have that
\begin{align*}
\tr\big(\chi_{\mathsf{N}_L}&[h(A_{C})-h(A)]\big)=L\,\sum\limits_{j=1}^2\tfrac{|E^{(j)}|}{2}\tr\big(\chi_{S^{(j)}}\big[h(A_{H^{(j)}})-h(A)\big]\big)
\nonumber\\
&+\tr\big(\chi_{C}\big[h(A_{C})-h(A_{H^{(1)}})-h(A_{H^{(2)}})+h(A)\big]\big)\nonumber\\
&-\cot(\gamma)\sum\limits_{j=1}^2\tr\big(\chi_{S^{(j)}} M(x\cdot \nu_{H^{(j)}})\big[h(A_{H^{(j)}})-h(A)\big]\big)+\mathcal{O}(L^{-\infty}),
\end{align*}
as $L\to\infty$.
\end{proposition}
\begin{proof}
We write
\begin{align}\label{eq:applyregularisationconvex}
\tr\big(\chi_{\mathsf{N}_L}\big[h(A_{C})-h(A)\big]\big)&=\tr\big(\chi_{\mathsf{N}_L}\big[h(A_{C})-h(A_{H^{(1)}})-h(A_{H^{(2)}})+h(A)\big]\big)\nonumber\\
&+\sum\limits_{j=1}^2\tr\big(\chi_{\mathsf{N}_L}\big[h(A_{H^{(j)}})-h(A)\big]\big).
\end{align}
Proposition \ref{prop:regularisationcornerop} implies that the operator
\begin{align*}
\chi_{C}\big[h(A_{C})-h(A_{H^{(1)}})-h(A_{H^{(2)}})+h(A)\big]
\end{align*}
is trace class with 
\begin{align*}
\tr\big(\chi_{C\setminus \mathsf{N}_L}\big[h(A_{C})-h(A_{H^{(1)}})-h(A_{H^{(2)}})+h(A)\big]\big)=\mathcal{O}(L^{-\infty}),
\end{align*}
since $\dist(0,C\setminus \mathsf{N}_L)\gtrsim L$.
Thus it remains to find the asymptotics for
\begin{align*}
\tr\big(\chi_{\mathsf{N}_L}\big[h(A_{H^{(j)}})-h(A)\big]\big), \ j=1,2.
\end{align*}
Recall the definition \eqref{eq:defGammajX} of the sectors $\Gamma^{(j)}$
and define its finite sections
\begin{align*}
\Gamma^{(j)}[r]:=\lbrace y \in \Gamma^{(j)}:y\cdot \nu^{(j)}\leq r\rbrace,\ j=1,2,\ r\geq 0.
\end{align*}
Applying the definition of $\mathsf{N}$, see \eqref{eq:defN}, and Corollary \ref{cor:Ldependenttrnormest} we get that
\begin{align*}
\tr\big(\chi_{\mathsf{N}_L}\big[h(A_{H^{(j)}})&-h(A)\big]\big)=\tr\big(\chi_{T_L^{(j)}}\big[h(A_{H^{(j)}})-h(A)\big]\big)\nonumber\\
&+\sgn(\gamma-\tfrac{\pi}{2}) \tr \big(\chi_{\Gamma^{(j)}[\epsilon L]}\big[h(A_{H^{(j)}})-h(A)\big]\big) + \mathcal{O}(L^{-\infty}).
\end{align*}
Furthermore, Lemma \ref{lem:Lterm} and another application of Corollary \ref{cor:Ldependenttrnormest} yield that
\begin{align*}
\tr\big(\chi_{\mathsf{N}_L}\big[h(A_{H^{(j)}})&-h(A)\big]\big)=\tfrac{L|E^{(j)}|}{2}\tr\big(\chi_{S^{(j)}}\big[h(A_{H^{(j)}})-h(A)\big]\big)\nonumber\\
&+\sgn(\gamma-\tfrac{\pi}{2})\tr\big(\chi_{\Gamma^{(j)}}\big[h(A_{H^{(j)}})-h(A)\big]\big)+\mathcal{O}(L^{-\infty}).
\end{align*}
Hence, the claim follows from Lemma \ref{lem:notrectangularpart} and \eqref{eq:applyregularisationconvex}. 
\end{proof}
\begin{proposition}[Concave corners]\label{prop:conccorners}
Let $X\in\Xi_>(P_L)$. Then we have that
\begin{align*}
\tr\big(\chi_{\mathsf{N}_L}&[h(A_{C})-h(A)]\big)=L\,\sum\limits_{j=1}^2\tfrac{|E^{(j)}|}{2}\tr\big(\chi_{S^{(j)}}\big[h(A_{H^{(j)}})-h(A)\big]\big)\nonumber\\
&+\tr\big(\chi_{H^{(1)}\cap H^{(2)}}\big[h(A_{C})-h(A)\big]\big)\nonumber\\[2ex]
&+\tr\big(\chi_{C\setminus H^{(1)}}\big[h(A_{C})-h(A_{H^{(2)}}\big]\big)\nonumber\\[2ex]
&+ \tr\big(\chi_{C\setminus H^{(2)}}\big[h(A_{C})-h(A_{H^{(1)}}\big]\big)\nonumber\\
&-\cot(\gamma)\sum\limits_{j=1}^2\tr(\chi_{S^{(j)}} M(x\cdot \nu_{H^{(j)}})\big[h(A_{H^{(j)}})-h(A)\big]\big)+\mathcal{O}(L^{-\infty}).
\end{align*}
as $L\to\infty$.
\end{proposition}
\begin{proof}
The proof is analogous to the convex case. We write
\begin{align*}
\tr\big(\chi_{\mathsf{N}_L}[h(A_{C})-h(A)]\big)=&\eta_1(L)+\eta_2(L),
\end{align*}
with
\begin{align*}
\eta_1(L)&:=\tr\big(\chi_{\mathsf{N}_L\cap H^{(1)}\cap H^{(2)}}\big[h(A_C)-h(A)\big]\big)+\tr\big(\chi_{\mathsf{N}_L\cap C\setminus H^{(1)}}\big[h(A_C)-h(A_{H^{(2)}})\big]\big)\nonumber\\[1ex]
&\ \ \ + \tr\big(\chi_{\mathsf{N}_L\cap C\setminus H^{(2)}}\big[h(A_C)-h(A_{H^{(1)}})\big]\big),
\end{align*}
and 
\begin{align*}
\eta_2(L)&:=\tr\big(\chi_{\mathsf{N}_L\cap C\setminus H^{(1)}}\big[h(A_{H^{(2)}})-h(A)\big]+\tr\big(\chi_{\mathsf{N}_L\cap C\setminus H^{(2)}}\big[h(A_{H^{(1)}})-h(A)\big].
\end{align*}
Proposition \eqref{prop:regularisationcornerop} implies that 
\begin{align*}
\eta_1(L)&=\tr\big(\chi_{H^{(1)}\cap H^{(2)}}\big[h(A_C)-h(A)\big]\big)+\tr\big(\chi_{C\setminus H^{(1)}}\big[h(A_C)-h(A_{H^{(2)}})\big]\big)\nonumber\\[1ex]
&\ \ \ + \tr\big(\chi_{C\setminus H^{(2)}}\big[h(A_C)-h(A_{H^{(1)}})\big]\big) + \mathcal{O}(L^{-\infty}).
\end{align*}
Moreover, we notice that the sectors $C\setminus H^{(j)}$, $j=1,2$, have an interior angle of $\gamma-\pi\in (0,\pi)$. This and the fact that $\cot(\gamma-\pi)=\cot(\gamma)$ explains why the contribution of $\eta_2(L)$ to the asymptotics is the same as in the convex case. Alternatively, one easily gets that, for instance,
\begin{align*}
\tr\big(\chi_{\mathsf{N}_L\cap C\setminus H^{(1)}}\big[h(A_{H^{(2)}})&-h(A)\big]\big)=\tr\big(\chi_{T_L^{(2)}}\big[h(A_{H^{(2)}})-h(A)\big]\big)\nonumber\\
&+\sgn(\gamma-\tfrac{3\pi}{2})\tr\big(\chi_{\mathsf{N}_L\cap \Gamma^{(2)}}\big[h(A_{H^{(2)}})-h(A)\big]\big).
\end{align*}
Thus, as in the convex case the claim follows from Corollaries \ref{cor:traceclassifpolybdd} and \ref{cor:Ldependenttrnormest}, and Lemmas \ref{lem:Lterm} and \ref{lem:notrectangularpart}.
\end{proof}
The proof of Theorem \ref{thm:abstractasympt} is now complete:
\begin{proof}[Proof of Theorem \ref{thm:abstractasympt}]
Subsection \ref{subsec:proofofabstrthm1} implies that for $h=h_1$, 
\begin{align*}
\tr\big(\chi_{P_L}\big[h(A_{P_L})-h(A)\big]\chi_{P_L}\big)&=\sum\limits_{X\in\Xi(P)}\tr\big(\chi_{\mathsf{N}_L(X)}\big[h(A_{C(X)})-h(A)\big]\big) + \mathcal{O}(L^{-\infty}).
\end{align*}
Hence, it follows from Propositions \ref{prop:convcorners} and \ref{prop:conccorners} that
\begin{align*}
\tr\big(\chi_{P_L}\big[h(A_{P_L})-h(A)\big]\chi_{P_L}\big)&=Lc_1+c_0+\mathcal{O}(L^{-\infty}).
\end{align*}
In view of \eqref{eq:leadingorderasympt}, this finishes the proof of the theorem.
\end{proof}

\section{Proof of Theorem~\ref{thm:coefficients}}\label{sec:evalcoeff}
It suffices to prove the theorem for test functions $h$ of the form $h(z)=\sum\limits_{k=2}^\infty a_kz^k$ since both sides of \eqref{eq:a1nuEasWienerHopf} and \eqref{eq:a0nuEasWienerHopf} vanish for linear functions $h$. Moreover, we may assume after a suitable rotation that $H_E=H=\mathbb{R}\times [0,\infty)$, i.e.~$S_E=S=[0,1]\times [0,\infty)$. Thus, we have that 
\begin{align*}
\sigma_{E,t}(\xi)=\sigma(t,\xi)=:\sigma_t(\xi),\ (t,\xi)\in\mathbb{R}^2.
\end{align*}
Define, for $\alpha\in \lbrace 0,1\rbrace$ and fixed $t\in\mathbb{R}$, the operator
\begin{align*}
B_\alpha(t):=M(x^\alpha)\big[h\lbrace W(\sigma_t)\rbrace -W(h\circ \sigma_t)\big],
\end{align*}
which acts on $L^2(\mathbb{R})$. Proposition \ref{prop:trnormdecay} implies that, for $\alpha\in\lbrace 0,1\rbrace$ and $t\in\mathbb{R}$,
\begin{align*}
\big\|B_\alpha(t)\big\|_1&\leq \sum\limits_{n=1}^\infty \big\|M(x^\alpha)\chi_{[n-1,n]}\big[h\lbrace A_{[0,\infty)}(\sigma_t)\rbrace-h\lbrace A(\sigma_t)\rbrace\big]\chi_{[0,\infty)}\big\|_1 \nonumber\\
&\leq \sum\limits_{n=1}^\infty \big\|M(x^\alpha)\chi_{[n-1,n]}\big\|\big\|\chi_{[n-1,n]}\big[h\lbrace A_{[0,\infty)}(\sigma_t)\rbrace-h\lbrace A(\sigma_t)\rbrace\big]\big\|_1 \nonumber\\
&\leq \sum\limits_{n=1}^\infty n^\alpha \ang{n-1}^{-3} \sum\limits_{k=2}^\infty k|a_k| \big[C_3 \|\sigma_t\|_{12}\big]^k\nonumber\\
&\lesssim \sum\limits_{k=2}^\infty k|a_k|\big[C_3 \|\sigma_t\|_{12}\big]^k<\infty.
\end{align*}
Hence, in view of Lemma \ref{lem:sigma} we have that $(t\mapsto \|B_\alpha(t)\|_1)\in L^1(\mathbb{R})\cap L^\infty(\mathbb{R})$. In particular, the right-hand sides of \eqref{eq:a1nuEasWienerHopf} and \eqref{eq:a0nuEasWienerHopf} are well-defined under our assumptions on $h$ and $\sigma$.

Introduce the unitary (identification) map
\begin{align*}
J:\mathsf{L}^2(\mathbb{R}^2)\to\mathsf{L}^2\big(\mathbb{R},\mathsf{L}^2(\mathbb{R})\big),\ (Jf)(t):=f(t,\,\cdot\,).
\end{align*}
Moreover, define the partial Fourier transforms $\mathcal{F}_1$, $\mathcal{F}_2$ on $\mathsf{L}^2(\mathbb{R}^2)$ that only act on the first and second variable, respectively.
To obtain the identities \eqref{eq:a1nuEasWienerHopf} and \eqref{eq:a0nuEasWienerHopf}, we first prove that
\begin{align}\label{eq:unitaryeqivdirectintegral}
M(x_2^\alpha) \big[h(A_H)-\chi_H h(A)\chi_H\big]=\mathcal{F}_1^\ast J^\ast B_\alpha J \mathcal{F}_1,
\end{align}
where $B_\alpha:=\int\limits_{\mathbb{R}}^\oplus dt\, B_\alpha(t)$ acts on $\mathsf{L}^2\big(\mathbb{R},\mathsf{L}^2(\mathbb{R})\big)$. For an introduction to direct integral operators see for example \cite{ReedSimon4}. To verify \eqref{eq:unitaryeqivdirectintegral}, notice that
\begin{align*}
\mathcal{F}_1\chi_H =\chi_H\mathcal{F}_1,
\end{align*} 
hence
\begin{align*}
A_H=\chi_H\mathcal{F}_1^\ast \mathcal{F}_2^\ast \sigma \mathcal{F}_2\mathcal{F}_1\chi_H=\mathcal{F}_1^\ast \chi_H\mathcal{F}_2^\ast \sigma \mathcal{F}_2\chi_H\mathcal{F}_1.
\end{align*}
Moreover, the definition of $J$ yields that
\begin{align*}
\chi_H\mathcal{F}_2^\ast\sigma \mathcal{F}_2\chi_H =J^\ast \int\limits_\mathbb{R}^\oplus dt\, W(\sigma_t)\, J,
\end{align*}
implying that
\begin{align}\label{eq:hAHfiberintegral}
h(A_H)&=
\mathcal{F}_1^\ast J^\ast h\big(\int\limits_\mathbb{R}^\oplus dt\, W(\sigma_t)\big)J\mathcal{F}_1 =\mathcal{F}_1^\ast J^\ast \int\limits_\mathbb{R}^\oplus dt\, h\lbrace W(\sigma_t)\rbrace J\mathcal{F}_1.
\end{align}
Similarly, one gets that
\begin{align}\label{eq:chiHhAchiHfiber}
\chi_H h(A)\chi_H = A_H(h\circ \sigma)=\mathcal{F}_1^\ast J^\ast \int\limits_\mathbb{R}^\oplus dt\, W(h\circ \sigma_t)\,J\mathcal{F}_1.
\end{align}
Thus, combining \eqref{eq:hAHfiberintegral} and \eqref{eq:chiHhAchiHfiber} gives
\begin{align*}
M(x_2^\alpha)\chi_H\big[h(A_H)-h(A)\big]\chi_H&=M(x_2^\alpha) \mathcal{F}_1^\ast J^\ast \int\limits_\mathbb{R}^\oplus dt\, B_0(t) J\mathcal{F}_1 =\mathcal{F}_1^\ast J^\ast \int\limits_\mathbb{R}^\oplus dt\, B_\alpha(t)\, J\mathcal{F}_1,
\end{align*}
which proves \eqref{eq:unitaryeqivdirectintegral}.

As a consequence of \eqref{eq:unitaryeqivdirectintegral}, the coefficients $a_1(\nu_E)$ and $a_0(\nu_E)$ are given by the traces of the operators $\chi_S \tilde{B}_\alpha\chi_S$, $\alpha=0,1$, where
\begin{align*}
\tilde{B}_\alpha:=\mathcal{F}_1^\ast J^\ast B_\alpha J\mathcal{F}_1.
\end{align*}
In order to calculate these traces, we evaluate the quadratic form of $\tilde{B}_\alpha$ on product states. Namely, for $\phi,\psi\in\mathsf{L}^2(\mathbb{R})$, we have that
\begin{align}\label{eq:quadraticformbalpha}
\big\langle \phi \otimes \psi,\tilde{B}_\alpha(\phi\otimes \psi)\big\rangle_{\mathsf{L}^2(\mathbb{R}^2)}&=\big\langle J((\mathcal{F}\phi)\otimes \psi), B_\alpha J ((\mathcal{F}\phi)\otimes \psi)\big\rangle_{\mathsf{L}^2(\mathbb{R},\mathsf{L}^2(\mathbb{R}))}\nonumber\\
&=\int\limits_\mathbb{R}dt\, \big\langle (\mathcal{F}\phi)(t)\psi, (\mathcal{F}\phi)(t)B_\alpha(t)\psi\big\rangle_{\mathsf{L}^2(\mathbb{R})} \nonumber\\
&=\int\limits_{\mathbb{R}}dt\, |(\mathcal{F}\phi)(t)|^2\langle \psi, B_\alpha(t)\psi \rangle_{\mathsf{L}^2(\mathbb{R})}.
\end{align}
Choose now an orthonormal basis $\lbrace \psi_n\rbrace_{n\in\mathbb{N}}$ of $\mathsf{L}^2(\mathbb{R})$, such that $\lbrace \psi_n\otimes\psi_m \rbrace_{n,m\in\mathbb{N}}$ is an orthonormal basis of $\mathsf{L}^2(\mathbb{R}^2)$. Then \eqref{eq:quadraticformbalpha} implies that
\begin{align*}
\tr\big(\chi_S \tilde{B}_\alpha \chi_S\big)&=\sum\limits_{n,m\in\mathbb{N}} \langle \psi_n\otimes \psi_m,\chi_S\tilde{B}_\alpha \chi_S \psi_n\otimes\psi_m\rangle_{\mathsf{L}^2(\mathbb{R}^2)}\nonumber\\
&=\sum\limits_{n,m\in\mathbb{N}}\int\limits_\mathbb{R}dt\, |\mathcal{F}(\chi_{[0,1]}\psi_n)(t)|^2\langle \psi_m, B_\alpha(t)\psi_m\rangle_{\mathsf{L}^2(\mathbb{R})}.
\end{align*}
As we have the estimate
\begin{align*}
\sum\limits_{m\in\mathbb{N}}\big|\langle \psi_m, B(t)\psi_m\rangle_{\mathsf{L}^2(\mathbb{R})}\big|\leq \|B_\alpha(t)\|_1\in \mathsf{L}^\infty(\mathbb{R}),
\end{align*}
we may apply Fubini's theorem to get that
\begin{align*}
\tr(\chi_S\tilde{B}_\alpha\chi_S)&=\sum\limits_{n\in\mathbb{N}}\int\limits_\mathbb{R}dt\, |\mathcal{F}(\chi_{[0,1]}\psi_n)(t)|^2 \tr B_\alpha(t) \nonumber\\
&=\sum\limits_{n\in\mathbb{N}}\langle \psi_n,\chi_{[0,1]}\mathcal{F}^\ast \tr B_\alpha(\,\cdot\,)\mathcal{F}\chi_{[0,1]}\psi_n\rangle_{\mathsf{L}^2(\mathbb{R})}.
\end{align*}
Hence, employing the fact that $\tr B_\alpha(\,\cdot\,)\in\mathsf{L}^1(\mathbb{R})$, we arrive at
\begin{align*}
\tr(\chi_S\tilde{B}_\alpha\chi_S)&=\tr\big(\chi_{[0,1]}\mathcal{F}^\ast \tr B_\alpha(\,\cdot\,)\mathcal{F}\chi_{[0,1]}\big)=(\tr B_\alpha\widecheck{)\,}\!(0)=\frac{1}{2\pi}\int\limits_\mathbb{R}dt\,\tr B_\alpha(t).
\end{align*}
This finishes the proof of Theorem \ref{thm:coefficients}.

\section{Radially symmetric symbols -- Proof of Theorem \ref{thm:rotsymmasymptsquare}}
\label{sec:rotsymmsymbols}

As in the statement of Theorem~\ref{thm:rotsymmasymptsquare} assume that the symbol $\sigma$ is radially symmetric and the test function $h$ is a quadratic polynomial, i.e.~$h(z)=z^2+bz$ for some $b\in\mathbb{C}$. The coefficient $c_2=c_2(P,h,\sigma)$ is easily computed from Theorem \ref{thm:abstractasympt}. Recall also that the linear part of $h$ does not contribute to the coefficients $c_1$ and $c_0$, so we may assume in the following that $h(z)=z^2$. To compute $c_1$ and $a_0(\nu_E)$, $E\in\mathcal{E}(P)$, we apply Theorem~\ref{thm:coefficients}. This is done in the next lemma.
\begin{lemma}\label{lem:halfspacesquared}
Let $h(z)=z^2$ and assume that $\sigma\in\mathsf{W}^{\infty,1}(\mathbb{R}^2)$ is radially symmetric. Then the coefficients $c_1$, $a_0(\nu_E)$ in Theorem \ref{thm:abstractasympt} satisfy the equations
\begin{align*}
c_1&=-2\left\vert{\partial P}\right\vert\int\limits_0^\infty dr\,r^2\check{\sigma}(r)^2, \\
\sum\limits_{E\in\mathcal{E}(P)}F(E)\, a_0(\nu_E)&=\sum\limits_{X\in\Xi(P)}\frac{\pi}{2}\cot(\gamma_X)\int\limits_0^\infty dr\, r^3 \check{\sigma}(r)^2.
\end{align*}
\end{lemma}
\begin{proof}
We first notice that the radial symmetry of the symbol implies that $\sigma_{E,t}(\xi)=\sigma(t,\xi)=\sigma_t(\xi)$ for all $E\in\mathcal{E}(P)$, and $t,\xi\in\mathbb{R}$. Furthermore, we make use of the formulas \eqref{eq:a1nuEasWienerHopf} and \eqref{eq:a0nuEasWienerHopf} in Theorem~\ref{thm:coefficients}. Similarly as in \cite{Widom1982}, one calculates that, for $\alpha\in\lbrace 0, 1\rbrace$, $t\in\mathbb{R}$,
\begin{align*}
-\tr\big(M(x^\alpha)\big[W(\sigma_t)^2-W(\sigma_t^2)\big]\big)&=\int\limits_0^\infty dx\,x^\alpha\!\int\limits_{-\infty}^0 dy\,\check{\sigma}_t(x-y)\check{\sigma}_t(y-x)\nonumber\\
&=\int\limits_0^\infty dx\, x^\alpha\! \int\limits_x^\infty dy\,\check{\sigma}_t(y)\check{\sigma}_t(-y)\nonumber\\
&=\int\limits_0^\infty dy\,\check{\sigma}_t(y) \check{\sigma}_t(-y)\int\limits_0^y dx\,x^\alpha\nonumber\\
&=\frac{1}{2}\int\limits_{-\infty}^\infty dy\, \frac{|y|^{\alpha+1}}{\alpha+1}\check{\sigma}_t(y)\check{\sigma}_t(-y).
\end{align*} 
Parseval's identity in the $t$-variable and the radial symmetry of $\check{\sigma}$ imply that
\begin{align*}
-\frac{1}{2\pi}\int\limits_{\mathbb{R}}dt\,\tr\big(M(x^\alpha)\big[W(\sigma_t)^2-W(\sigma_t^2)\big]\big)&=\frac{1}{4\pi}\int\limits_{\mathbb{R}}dt\,\int\limits_{\mathbb{R}}dy\,\frac{|y|^{\alpha+1}}{\alpha+1}\check{\sigma}_t(y)\check{\sigma}_t(-y)\nonumber\\
&=\frac{1}{2}\int\limits_{\mathbb{R}}dy_1\int\limits_{\mathbb{R}}dy_2 \frac{|y_2|^{\alpha+1}}{\alpha+1}\check{\sigma}(-y_1,y_2)\check{\sigma}(y_1,-y_2)\nonumber\\
&=\frac{1}{2}\int\limits_0^\infty dr\,r\int\limits_0^{2\pi} d\theta \frac{|r\sin(\theta)|^{\alpha+1}}{\alpha+1}\check{\sigma}(r)^2\nonumber\\
&=\begin{cases}
2\int\limits_0^\infty dr\, r^2\check{\sigma}(r)^2,& \ \alpha=0,\\
\frac{\pi}{4}\int\limits_0^\infty dr\, r^3\check{\sigma}(r)^2,&\ \alpha=1.
\end{cases}
\end{align*}
Hence, the claim follows from Theorem~\ref{thm:coefficients} and the definition of $F(E)$, see \eqref{eq:defF}.
\end{proof}
It remains to compute the coefficients $b_0(X)$, $X\in\Xi(P)$, from formulas \eqref{eq:b0Xconvex} and \eqref{eq:b0Xconcave}. This calculation is performed in the next lemma.
\begin{lemma}\label{lem:cornersquare}
Let $h(z)=z^2$ and assume that the symbol $\sigma\in \mathsf{W}^{\infty,1}(\mathbb{R}^2)$ is radially symmetric.
\\
Then for every $X\in\Xi(P)$ the formula
\begin{align}\label{eq:b0Xexplicitconvcorners}
b_0(X)=\frac{1-\gamma_X\cot(\gamma_X)}{2}\int\limits_0^\infty dr\, r^3\check{\sigma}(r)^2
\end{align}
holds.
\end{lemma}
\begin{proof}
Fix $X\in\Xi(P)$ and omit the subscript or argument ``$(X)$'' for the duration of the proof. As usual, we treat the cases of convex and concave corners separately.

First, let $X\in\Xi_<(P)$. Then, due to the radial symmetry of $\sigma$, we may assume that
\begin{align}\label{eq:Cstartsfromy2axis}
C=\lbrace (r\cos(\theta),r\sin(\theta)):\ r\geq 0,\  \theta\in [0,\gamma]\rbrace,
\end{align}
with $\gamma\in (0,\pi)$. From \eqref{eq:b0Xconvex} one gets that
\begin{align}\label{eq:squarecomputedconvexcorner}
b_0&=\tr\big[\chi_{C} \big([A_{C}]^2-[A_{H^{(1)}}]^2-[A_{H^{(2)}}]^2+A^2\big)\big] =\tr\big(\chi_{C}A\chi_{-C}A\chi_C\big),
\end{align}
and evaluating the trace gives
\begin{align*}
b_0&=\int\limits_{C}dx \int\limits_{-C}dy\,\check{\sigma}(x-y)^2=\int\limits_C dx \!\int\limits_{x+C}\! dy\,\check{\sigma}(y)^2=\int\limits_C dy\, \check{\sigma}(y)^2 \left\vert{(y-C)\cap C}\right\vert.
\end{align*}
For the last equality we have used the fact that $x\in C$ and $y\in x+C$ is equivalent to $y\in C$ and $x\in (y-C)\cap C$. Applying \eqref{eq:Cstartsfromy2axis} and the assumption that $\gamma\in(0,\pi)$, one easily computes that, for $y\in C$,
\begin{align*}
\left\vert{(y-C)\cap C}\right\vert&=y_1y_2-\cot(\gamma)y_2^2.
\end{align*}
Hence, the radial symmetry of $\sigma$ yields
\begin{align*}
\int\limits_C dy\, \check{\sigma}(y)^2 \left\vert{(y-C)\cap C}\right\vert&=\int\limits_0^\infty dr\, r^3 \check{\sigma}(r)^2\int\limits_0^\gamma d\theta\, \cos(\theta)\sin(\theta)-\cot(\gamma)\sin^2(\theta) \nonumber\\
&=\frac{1-\gamma\cot(\gamma)}{2}\int\limits_0^\infty dr\, r^3 \check{\sigma}(r)^2,
\end{align*}
and the claim follows for $X\in\Xi_<(P)$.

Secondly, let $X\in\Xi_>(P)$. Then we get from \eqref{eq:b0Xconcave} that
\begin{align*}
b_0&=\tr\big[\chi_{H^{(1)}\cap H^{(2)}}\big([A_{C}]^2-A^2\big)\big]+\tr\big[\chi_{C\setminus H^{(1)}}\big([A_{C}]^2-[A_{H^{(2)}}]^2\big)\big]\nonumber\\[2ex]
&\ \ \ +\tr\big[\chi_{C\setminus H^{(2)}}\big([A_{C}]^2-[A_{H^{(1)}}]^2\big)\big]\nonumber\\
&=-\tr\big(\chi_{H^{(1)}\cap H^{(2)}}A\chi_{-H^{(1)}\cap H^{(2)}}A\chi_{H^{(1)}\cap H^{(2)}}\big)+\sum\limits_{j=1}^2\tr\big(\chi_{C\setminus H^{(j)}}A\chi_{-C\setminus H^{(j)}}A\chi_{C\setminus H^{(j)}}\big).
\end{align*}
Note that $H^{(1)}\cap H^{(2)}$ and $C\setminus H^{(j)}$, $j=1,2$, are convex sectors with interior angles $2\pi-\gamma$ and $\gamma-\pi$, respectively. Thus, the formulas \eqref{eq:squarecomputedconvexcorner} and \eqref{eq:b0Xexplicitconvcorners} for $X\in\Xi_<(P)$ yield
\begin{align*}
b_0&=\Big[-\frac{1-(2\pi-\gamma)\cot(2\pi-\gamma)}{2}+1-(\gamma-\pi)\cot(\gamma-\pi)\Big]\int\limits_0^\infty dr\,r^3 \check{\sigma}(r)^2\nonumber\\
&=\frac{1-\gamma\cot(\gamma)}{2}\int\limits_0^\infty dr\, r^3\check{\sigma}(r)^2.
\end{align*}
This finishes the proof of the lemma.
\end{proof}
Theorem \ref{thm:rotsymmasymptsquare} follows now from combining Lemmas~\ref{lem:halfspacesquared} and \ref{lem:cornersquare}. 
\begin{appendix}
\section*{Appendix}\label{sec:appendix}
\renewcommand{\theequation}{A.\arabic{equation}}
\renewcommand{\thetheorem}{A.\arabic{theorem}}
\setcounter{equation}{0}
\setcounter{theorem}{0}
The purpose of this appendix is to provide a proof of the following result.
\begin{lemma}\label{lem:appendix}
Suppose that $d\geq 2$ and let $\Omega\subset\mathbb{R}^d$ be a bounded set with smooth boundary. Moreover, assume that $\sigma\in\mathsf{W}^{\infty,1}(\mathbb{R}^d)$ and let $h(z)=z^2+bz$ for some $b\in\mathbb{C}$. Then the coefficient $\mathcal{B}_{d-2}=\mathcal{B}_{d-2}(\Omega,h,\sigma)$ in \eqref{eq:trasymptsmoothboundary} vanishes:
\begin{align*}
\mathcal{B}_{d-2}(\Omega,h,\sigma)=0.
\end{align*}
\end{lemma}

For $\sigma$ and $\Omega$ as in the lemma and (general) entire test functions $h$ with $h(0)=0$, a formula for $\mathcal{B}_{d-2}$ is contained, for instance, in \cite{Roccaforte1984}. In order to write it down, we need to fix some notation. Let $d\Sigma$ denote the surface measure on $\partial\Omega$ and write $\nu_x$ for the inwards pointing unit normal vector at $x\in\partial\Omega$. Consider the canonical volume element $dX=d\Sigma d\overline{\xi}$ on $T^{\ast}(\partial\Omega)$ where $d\overline{\xi}$ is the Lebesgue measure on $\lbrace\nu_x\rbrace^{\perp}$. Moreover, let $L$ denote the second fundamental form on $\partial\Omega$ with respect to the unit normal $\nu$ and write $H$ for $(d-1)$ times the mean curvature on $\partial\Omega$. Finally, introduce for a vector $w\in\mathbb{R}^d$ its orthogonal projection $w_{T_x}=w_{T_x(\partial\Omega)}$ onto $T_x(\partial\Omega)=\lbrace \nu_x \rbrace^{\perp}$.

In view of \cite[Thm. 1.1]{Roccaforte1984} the coefficient $\mathcal{B}_{d-2}=\mathcal{B}_{d-2}(\Omega,h,\sigma)$ is given by
\begin{align}\label{eq:Bd-2}
&\mathcal{B}_{d-2}=-\frac{1}{2(2\pi)^{d+2}}\!\!\int\limits_{T^\ast(\partial\Omega)}\!\!\!\!dX\int\limits_{\mathbb{R}}d\xi_1\int\limits_{\mathbb{R}}\frac{d\xi_2}{\xi_1-\xi_2}\int\limits_{\mathbb{R}}\frac{d\xi_3}{\xi_1-\xi_3} \Bigg\lbrace\sum\limits_{k=1}^3\frac{h(\sigma(\overline{\xi}+\xi_k\nu))}{\prod\limits_{j\neq k}[\sigma(\overline{\xi}+\xi_k\nu)-\sigma(\overline{\xi}+\xi_j\nu)]}\Bigg\rbrace\nonumber\\
&\quad\times\Big\lbrace L\Big[(\nabla\sigma)_T(\overline{\xi}+\xi_2\nu),(\nabla\sigma)_T(\overline{\xi}+\xi_3\nu)\Big]-H\big[\nu\cdot(\nabla\sigma)(\overline{\xi}+\xi_2\nu)\big]\big[\nu\cdot(\nabla\sigma)(\overline{\xi}+\xi_3\nu)\big]\Big\rbrace,
\end{align}
where the integrals over $\xi_2$ and $\xi_3$ are interpreted as Cauchy principal values. Equipped with this formula, we are ready to prove the lemma.
\begin{proof}[Proof of Lemma \ref{lem:appendix}]
Note that, for the given function $h$, one has that
\begin{align*}
\sum\limits_{k=1}^3\frac{h(\sigma(\overline{\xi}+\xi_k\nu))}{\prod\limits_{j\neq k}[\sigma(\overline{\xi}+\xi_k\nu)-\sigma(\overline{\xi}+\xi_j\nu)]}=1,
\end{align*}
for all $\overline{\xi},\nu\in\mathbb{R}^d$ and $\xi_1,\xi_2,\xi_3\in\mathbb{R}$. Thus, as the Hilbert transform 
\begin{align*}
\mathsf{C}^{\infty}(\mathbb{R})\cap\mathsf{L}^2(\mathbb{R})\ni f\mapsto \tilde{f};\quad
\tilde{f}(t):=\frac{1}{\pi}\lim\limits_{\epsilon \searrow0}\int\limits_{|s-t|>\epsilon}\!\!\!\!ds \,\frac{f(s)}{t-s},
\end{align*}
extends to a unitary operator on $\mathsf{L}^2(\mathbb{R})$, the formula \eqref{eq:Bd-2} for $\mathcal{B}_{d-2}$ simplifies to 
\begin{align}\label{eq:Bd-2simplified}
-\frac{1}{8(2\pi)^d}\!\!\!\!\int\limits_{T^\ast(\partial\Omega)}\!\!\!\!\! dX\int\limits_{\mathbb{R}} d\zeta\, \Big\lbrace L\Big[(\nabla\sigma)_T(\overline{\xi}+\zeta\nu),(\nabla\sigma)_T(\overline{\xi}+\zeta\nu)\Big]-H\big[\nu\cdot(\nabla\sigma)(\overline{\xi}+\zeta\nu)\big]^2\Big\rbrace.
\end{align}
To see that this expression vanishes identically we repeat an argument from \cite[p. 600]{Roccaforte1984}. Writing out the volume element $dX=d\Sigma d\overline{\xi}$ and combining the $\overline{\xi}$- and $\zeta$-integration in \eqref{eq:Bd-2simplified}, one arrives at
\begin{align*}
\mathcal{B}_{d-2}=-\frac{1}{8(2\pi)^d}\!\!\int\limits_{\partial\Omega} d\Sigma(x)\int\limits_{\mathbb{R}^d} d\xi\,\Big\lbrace L_x\big[(\nabla\sigma)_T(\xi),(\nabla\sigma)_T(\xi)\Big]-H_x\big[\nu_x\cdot(\nabla\sigma)(\xi)\big]^2\Big\rbrace.
\end{align*}
Hence, the lemma follows from Fubini's theorem and the identity
\begin{align*}
\int\limits_{\partial\Omega}d\Sigma(x) \Big(L_x\big[w_{T_x},w_{T_x}\big]-H(x)[\nu_x\cdot w]^2 \Big)=0,
\end{align*}
which holds for any $w\in\mathbb{R}^d$, see \cite[Eq. (4.16)]{Roccaforte1984}.
\end{proof}
\end{appendix}

\Addresses

\begin{thebibliography}{10}
\providecommand{\url}[1]{\texttt{#1}}
\providecommand{\urlprefix}{URL }
\providecommand{\eprint}[2][]{\url{#2}}

\bibitem{BS}
M.~S. Birman and M.~Z. Solomjak, \emph{Spectral Theory of Selfadjoint Operators
  in {H}ilbert Space}. Mathematics and its Applications (Soviet Series), D.
  Reidel, 1987. Translated from the 1980 Russian original by S. Khrushch{\"e}v
  and V. Peller.

\bibitem{Brislawn1988}
C.~Brislawn, \emph{Kernels of Trace Class Operators}. Proc. Amer. Math. Soc.
  \textbf{104(4)}: 1181--1190, 1988.

\bibitem{Demengel2012}
F.~Demengel and G.~Demengel, \emph{Functional spaces for the theory of elliptic
  partial differential equations}. Universitext, Springer, London; EDP
  Sciences, Les Ulis, 2012. Translated from the 2007 French original by Reinie
  Ern\'e.

\bibitem{Dietlein2018}
A.~Dietlein, \emph{Full Szeg{\H{o}}-Type Trace Asymptotics for Ergodic
  Operators on Large Boxes}. Comm. Math. Phys. 2018.
  \mbox{doi}:\url{10.1007/s00220-018-3161-5}.

\bibitem{Doktorskii1984}
R.~Y. Doktorski\u\i, \emph{Generalization of the {S}zeg\H o limit theorem to
  the multidimensional case}. Sibirsk. Mat. Zh. \textbf{25(5)}: 20--29, 1984.

\bibitem{ElgartPastur2016}
A.~Elgart, L.~Pastur, and M.~Shcherbina, \emph{Large block properties of the
  entanglement entropy of free disordered fermions}. J. Statist. Phys.
  \textbf{166(3-4)}: 1092--1127, 2017.

\bibitem{GioevKlich2006}
D.~Gioev and I.~Klich, \emph{Entanglement Entropy of Fermions in Any Dimension
  and the Widom Conjecture}. Phys. Rev. Lett. \textbf{\textbf{96}}: 100503,
  2006.

\bibitem{Helling2009}
R.~Helling, H.~Leschke, and W.~Spitzer, \emph{{A Special Case of a Conjecture
  by Widom with Implications to Fermionic Entanglement Entropy}}. Int. Math.
  Res. Not. \textbf{\textbf{2011}}: 1451--1482, 2011.

\bibitem{Kateb2000}
D.~Kateb and A.~Seghier, \emph{Expansion of the inverse of positive-definite
  {T}oeplitz operators over polytopes}. Asymptot. Anal. \textbf{22(3-4)}:
  205--234, 2000.

\bibitem{KirschPastur2014}
W.~Kirsch and L.~Pastur, \emph{Analogues of {S}zeg{\H o}'s Theorem for Ergodic
  Operators}. Zap. Nauchn. Sem. S.-Peterburg. Otdel. Mat. Inst. Steklov. (POMI)
  \textbf{206(1)}: 103--130, 2015.

\bibitem{Klich2006}
I.~Klich, \emph{Lower Entropy Bounds and Particle Number Fluctuations in a
  Fermi Sea}. J. Phys. A: Math. Gen. \textbf{\textbf{39}(4)}: L85, 2006.

\bibitem{LandauWidom}
H.~J. Landau and H.~Widom, \emph{Eigenvalue Distribution of Time and Frequency
  Limiting}. J. Math. Anal. Appl. \textbf{\textbf{77}(2)}: 469--481, 1980.

\bibitem{LeschkeSobolevSpitzer2014}
H.~Leschke, A.~V. Sobolev, and W.~Spitzer, \emph{Scaling of R\'enyi
  Entanglement Entropies of the Free Fermi-Gas Ground State: A Rigorous Proof}.
  Phys. Rev. Lett. \textbf{\textbf{112}}: 160403, 2014.

\bibitem{LeschkeSobolevSpitzer2016}
H.~Leschke, A.~V. Sobolev, and W.~Spitzer, \emph{Large-Scale Behaviour of Local
  and Entanglement Entropy of the Free Fermi Gas at Any Temperature}. J. Phys.
  A: Math. Theor. \textbf{\textbf{49}(30)}: 30LT04, 2016.

\bibitem{LuRowlett2016}
Z.~Lu and J.~M. Rowlett, \emph{One can hear the corners of a drum}. Bull. Lond.
  Math. Soc. \textbf{48(1)}: 85--93, 2016.

\bibitem{MazzeoRowlett2015}
R.~Mazzeo and J.~M. Rowlett, \emph{A heat trace anomaly on polygons}. Math.
  Proc. Cambridge Philos. Soc. \textbf{159(2)}: 303–319, 2015.

\bibitem{PasturSlavin2014}
L.~Pastur and V.~Slavin, \emph{Area Law Scaling for the Entropy of Disordered
  Quasifree Fermions}. Phys. Rev. Lett. \textbf{\textbf{113}}: 150404, 2014.

\bibitem{PfirschSobolev2018}
B.~Pfirsch and A.~V. Sobolev, \emph{Formulas of Szeg{\H{o}} Type for the
  Periodic Schr{\"o}dinger Operator}. Comm. Math. Phys. \textbf{358(2)}:
  675--704, 2018.

\bibitem{ReedSimon4}
M.~Reed and B.~Simon, \emph{Methods of Modern Mathematical Physics. {IV}.
  {A}nalysis of Operators}. Academic Press [Harcourt Brace Jovanovich,
  Publishers], New York-London, 1978.

\bibitem{Rinkel2014}
J.-M. Rinkel and A.~Seghier, \emph{Factorisation de fonctions positives sur le
  tore: Applications {\`a} l'inverse des op{\'e}rateurs de Toeplitz
  tronqu{\'e}s}. Annales math{\'e}matiques du Qu{\'e}bec \textbf{38(2)}:
  189--230, 2014.

\bibitem{Roccaforte1984}
R.~Roccaforte, \emph{Asymptotic Expansions of Traces for Certain Convolution
  Operators}. Trans. Amer. Math. Soc. \textbf{285(2)}: 581--602, 1984.

\bibitem{Roccaforte2013}
R.~Roccaforte, \emph{Volume estimates and spectral asymptotics for a class of
  pseudo-differential operators}. J. Pseudo-Differ. Oper. Appl \textbf{4(1)}:
  25--43, 2013.

\bibitem{Seghier1986}
A.~Seghier, \emph{Inversion de la matrice de Toeplitz en d dimensions et
  développement asymptotique de la trace de l'inverse à l'ordre d}. J. Funct.
  Anal. \textbf{67(3)}: 380 -- 412, 1986.

\bibitem{Simon}
B.~Simon, \emph{Trace Ideals and Their Applications}, \emph{Mathematical
  Surveys and Monographs}, vol. 120. Second edn., American Mathematical
  Society, Providence, RI, 2005.

\bibitem{Sobolev2010}
A.~V. Sobolev, \emph{Quasi-Classical Asymptotics for Pseudodifferential
  Operators with Discontinuous Symbols: Widom's Conjecture}. Funct. Anal. Appl.
  \textbf{\textbf{44}(4)}: 313--317, 2010.

\bibitem{Sobolev2013}
A.~V. Sobolev, \emph{Pseudo-Differential Operators with Discontinuous Symbols:
  {W}idom's Conjecture}. Mem. Amer. Math. Soc. \textbf{\textbf{222}(1043)}:
  vi+104, 2013.

\bibitem{Sobolev2016}
A.~V. Sobolev, \emph{Functions of Self-Adjoint Operators in Ideals of Compact
  Operators}. J. Lond. Math. Soc. 2016.

\bibitem{Sobolev2017}
A.~V. Sobolev, \emph{Quasi-classical asymptotics for functions of Wiener--Hopf
  operators: smooth versus non-smooth symbols}. Geom. Funct. Anal.
  \textbf{27(3)}: 676--725, 2017.

\bibitem{Sobolev2018}
A.~V. {Sobolev}, \emph{{On the Szeg\H{o} formulas for truncated Wiener-Hopf
  operators}}. ArXiv e-prints 2018. \eprint{1801.02520}.

\bibitem{Szego1952}
G.~Szeg{\"o}, \emph{On Certain {H}ermitian Forms Associated with the {F}ourier
  Series of a Positive Function}. Comm. S\'em. Math. Univ. Lund, Tome
  Suppl\'ementaire 228--238, 1952.

\bibitem{Thorsen1996}
B.~H. Thorsen, \emph{An {$N$}-dimensional analogue of {S}zeg\H o's limit
  theorem}. J. Math. Anal. Appl. \textbf{198(1)}: 137--165, 1996.

\bibitem{Widom1960}
H.~Widom, \emph{A Theorem on Translation Kernels in n Dimensions}. Trans. Amer.
  Math. Soc. \textbf{94(1)}: 170--180, 1960.

\bibitem{WIDOM1980}
H.~Widom, \emph{Szeg\H{o}'s limit theorem: The higher-dimensional matrix case}.
  J. Funct. Anal. \textbf{39(2)}: 182 -- 198, 1980.

\bibitem{Widom1981}
H.~Widom, \emph{On a Class of Integral Operators with Discontinuous Symbol}.
  \emph{Toeplitz Centennial ({T}el {A}viv, 1981)}, \emph{Operator Theory: Adv.
  Appl.}, vol.~4, 477--500, Birkh\"auser, Basel-Boston, Mass., 1982.

\bibitem{Widom1982}
H.~Widom, \emph{A trace formula for {W}iener-{H}opf operators}. J. Operat.
  Theor. \textbf{8(2)}: 279--298, 1982.

\bibitem{Widom1985}
H.~Widom, \emph{Asymptotic Expansions for Pseudodifferential Operators on
  Bounded Domains}, \emph{Lecture Notes in Mathematics}, vol. 1152.
  Springer-Verlag, New York-Berlin, 1985.

\end{thebibliography}
\end{document}